\documentclass[11pt,twoside]{article}
\usepackage{amssymb,amsmath,amsthm,amsfonts,mathrsfs,hyperref}
\usepackage{times}
\usepackage{enumerate}
\usepackage{cite,titletoc}
\usepackage[toc,page,title,titletoc,header]{appendix}

\allowdisplaybreaks

\def\cA{\mathcal A}

\def\cC{\mathcal C}

\def\cF{\mathcal F}

\def\cM{\mathcal M}

\def\cP{\mathcal P}

\def\cU{\mathcal U}
\def\cV{\mathcal V}
\def\cW{\mathcal W}
\def\cX{\mathcal X}

\def\N{\mathop{\mathbb N\kern 0pt}\nolimits}
\def\Z{\mathop{\mathbb Z\kern 0pt}\nolimits}
\def\Q{\mathop{\mathbb Q\kern 0pt}\nolimits}
\def\R{\mathop{\mathbb R\kern 0pt}\nolimits}
\def\T{\mathop{\mathbb T\kern 0pt}\nolimits}
\def\SS{\mathop{\mathbb S\kern 0pt}\nolimits}

\def\ds{\displaystyle}
\def\f{\frac}

\def\supp{\mathop{\rm supp}\nolimits}

\def\p{\partial}

\def\ve{\varepsilon}
\def\vp{\varphi}

\def\supp{\operatorname{supp}}
\def\ls{\lesssim}

\newcommand{\w}[1]{\langle {#1} \rangle}

\hypersetup{colorlinks=true,linkcolor=blue,citecolor=red,urlcolor=cyan}

\theoremstyle{plain}
\newtheorem{theorem}{Theorem}[section]

\newtheorem{lemma}[theorem]{Lemma}
\newtheorem{corollary}[theorem]{Corollary}

\theoremstyle{definition}

\newtheorem{remark}{Remark}[section]

\numberwithin{equation}{section}

\title{Global existence and scattering of small data smooth solutions to a class of quasilinear
wave systems on $\mathbb{R}^2\times\mathbb{T}$}

\author{Fei Hou$^{1,*}$ \quad Fei Tao$^{2,*}$ \quad Huicheng Yin$^{3,}$
    \footnote{Hou Fei(\texttt{fhou$@$nju.edu.cn}),  Tao Fei(\texttt{feitao$@$njupt.edu.cn}) and  Yin Huicheng(\texttt{huicheng$@$nju.edu.cn}, \texttt{05407$@$njnu.edu.cn}) are supported by the NSFC (No.~12331007, No.~12101304). In addition, Tao Fei is supported by the Natural Science Research Start-up Foundation of Recruiting Talents of Nanjing University of Posts and Telecommunications (Grant No. NY223200);
     Yin Huicheng is supported by the National key research and development program of China (No.2020YFA0713803).}\\
    [12pt]{\small 1. Department of Mathematics, Nanjing University, Nanjing, 210093, China}\\
    {\small 2. School of Science, Nanjing University of Posts and Telecommunications,}\\
    {\small Nanjing, 210023, China}\\
    {\small 3. School of Mathematical Sciences and Mathematical Institute,}\\
    {\small Nanjing Normal University, Nanjing, 210023, China}}

\begin{document}

\date{}
\maketitle
\thispagestyle{empty}

\begin{abstract}
In this paper, we are concerned with the global existence and scattering of small data smooth solutions to a class of
quasilinear wave systems on the product space $\mathbb{R}^2\times\mathbb{T}$. These quasilinear wave systems include
3D irrotational potential flow equation of Chaplygin gases,
3D relativistic membrane equation, some 3D quasilinear wave equations which come
from the corresponding Lagrangian functionals as perturbations of the Lagrangian densities of linear waves,
and nonlinear wave maps system. Through looking for some suitable transformations of unknown functions, 
the nonlinear wave system can
be reduced into a more tractable form. Subsequently, by applying the vector-field method together with the ghost weight
technique as well as deriving some kinds of weighted $L^\infty-L^\infty$ and $L^\infty-L^2$ estimates of solution $w$  to the 2D 
linear wave equation $\Box w=f(t,x)$, the global existence and
scattering of small data solutions are established.

\vskip 0.2 true cm

\noindent
\textbf{Keywords.} Global existence, scattering, quasilinear wave system, $L^\infty-L^\infty$ estimate, 

\qquad \quad $L^\infty-L^2$ estimate, ghost weight

\vskip 0.2 true cm
\noindent
\textbf{2020 Mathematical Subject Classification.}  35L05, 35L15, 35L70
\end{abstract}

\vskip 0.2 true cm

\addtocontents{toc}{\protect\thispagestyle{empty}}
\tableofcontents

\section{Background and main results}
In this paper, we study the Cauchy problem of the following 3D quasilinear wave system on $\R^2\times\T$
\begin{equation}\label{QWE}
\left\{
\begin{aligned}
&\Box u^i=F^i(u,\p u, \p^2u):=\sum_{j,k=1}^m\sum_{|a|+|b|\le1}C^{ab}_{ijk}Q_0(\p^au^j,\p^bu^k)
+\sum_{j,k,l=1}^mC_{ijkl}Q_0(u^j,u^k)u^l\\
&\qquad\quad+\sum_{j,k,l=1}^m\sum_{\alpha,\beta,\mu,\nu=0}^3Q_{ijkl}^{\alpha\beta\mu\nu}\p^2_{\alpha\beta}u^j\p_{\mu}u^k\p_{\nu}u^l,\qquad i=1,\cdots,m,\\
&(u,\p_tu)(0,x,y)=(u_{(0)},u_{(1)})(x,y),
\end{aligned}
\right.
\end{equation}
where $u=(u^1, ..., u^m)$, $(t,x,y)\in [0,+\infty)\times\R^2\times\T$, $x=(x_1,x_2)$, $\ds\Box:=\Box_{t,x,y}=\p_t^2-\Delta_x-\p_y^2$, $\Delta_x=\p_1^2+\p_2^2$, $\p=(\p_0,\p_1,\p_2,\p_3)=(\p_t,\p_{x_1},\p_{x_2},\p_y)$, $\ds Q_0(f,g)=\p_tf\p_tg-\sum_{j=1}^3\p_jf\p_jg$, $\p^a=\p_0^{a_1}\p_1^{a_2}\p_2^{a_3}\p_y^{a_4}$ with multi-index $a\in\N_0^4$,
and $u_{(r)}=(u_{(r)}^1, ..., u_{(r)}^m)$ for $r=0,1$.
In addition, $(u_{(0)},u_{(1)})$ is $2\pi$-periodic in $y$, namely,
$(u_{(0)},u_{(1)})(x,y+2\pi)=(u_{(0)},u_{(1)})(x,y)$. Note that
the first term on the right hand side of the first line in \eqref{QWE} also admits such a form
\begin{equation}\label{NL:rewrite}
\sum_{|a|+|b|\le1}C^{ab}_{ijk}Q_0(\p^au^j,\p^bu^k)
=\sum_{\alpha,\beta=0}^3Q_{ijk}^{\alpha\beta}\p_{\alpha}u^j\p_{\beta}u^k
+\sum_{\alpha,\beta,\mu=0}^3Q_{ijk}^{\alpha\beta\mu}\p^2_{\alpha\beta}u^j\p_{\mu}u^k.
\end{equation}
In addition, the following symmetric conditions are imposed for $i,j,k,l=1,\cdots,m$,
\begin{equation}\label{sym:condition}
Q_{ijkl}^{\alpha\beta\mu\nu}=Q_{ijkl}^{\beta\alpha\mu\nu}=Q_{jikl}^{\alpha\beta\mu\nu},
\qquad Q_{ijk}^{\alpha\beta\mu}=Q_{ijk}^{\beta\alpha\mu}=Q_{jik}^{\alpha\beta\mu}.
\end{equation}

In order to obtain the global existence of small data solution $u$ of \eqref{QWE},
as in \cite{HTY24}, the related partial null condition is required  for the cubic nonlinearity in
\eqref{QWE}:

{\it For any $(\xi_0,\xi_1,\xi_2)\in\{\pm1\}\times\SS$ and $i,j,k,l=1,\cdots,m$, it holds that
\begin{equation}\label{part:null:condtion}
\sum_{\alpha,\beta,\mu,\nu=0}^2Q_{ijkl}^{\alpha\beta\mu\nu}\xi_{\alpha}\xi_{\beta}\xi_{\mu}\xi_{\nu}\equiv0.
\end{equation}}

On the other hand, with respect to \eqref{NL:rewrite}, it is easy to check that for $i,j,k,l=1,\cdots,m$ and $(\xi_0,\xi_1,\xi_2)\in\{\pm1\}\times\SS$,
\begin{equation}\label{NC:rewrite}
\sum_{\alpha,\beta=0}^2Q_{ijk}^{\alpha\beta}\xi_{\alpha}\xi_{\beta}\equiv0\quad\text{and}\quad
\sum_{\alpha,\beta,\mu=0}^2Q_{ijk}^{\alpha\beta\mu}\xi_{\alpha}\xi_{\beta}\xi_{\mu}\equiv0.
\end{equation}

There are a lot of physical or geometric models of the second order quasilinear wave equations
or systems such that their nonlinearities admit the forms in \eqref{QWE}.
For examples, these models include:

\vskip 0.2 true cm

(1) {\bf 3D irrotational potential flow equation of Chaplygin gases}

\vskip 0.1 true cm

The compressible isentropic Euler equations on $\R^2\times\T$ are
\begin{equation}\label{Euler}
\left\{
\begin{aligned}
&\p_t\rho+div(\rho v)=0,\\
&\p_t(\rho v)+div(\rho v \otimes v)+\nabla p=0,\\
\end{aligned}
\right.
\end{equation}
where $(t,x,y)\in [0,+\infty)\times\mathbb{R}^2\times\Bbb T$, $\nabla=(\p_1,\p_2, \p_3)=(\p_{x_1},\p_{x_2},\p_y)$, $v=(v_1,v_2,v_3)$,
$\rho$, $p$ stand for the velocity, density, pressure, respectively.
The state equation of Chaplygin gases (see \cite{CF}) is
\begin{equation}\label{pressure2}
p(\rho)=P_0-\frac{B}{\rho},
\end{equation}
where $P_0$ and $B$ are positive constants.
Suppose that \eqref{Euler} admits the following irrotational smooth initial data with $2\pi$-period on $y$:
\begin{equation}\label{irrot:condition}
\begin{split}
&(\rho,v)(0,x,y)=(\bar\rho+\rho_0(x,y),v_1^0(x,y),v_2^0(x,y),v_3^0(x,y)),\\
&\p_iv_j^0=\p_jv_i^0,\quad 1\le i<j\le3,
\end{split}
\end{equation}
where $\bar\rho+\rho_0(x,y)>0$, and $\bar\rho$ is
a positive constant which can be normalized so that the sound speed $c(\bar\rho)=\sqrt{p'(\bar\rho)}=1$. Then by the irrotationality
of \eqref{Euler}, one can
introduce a potential function $\phi$ such that $v=\nabla\phi$. Meanwhile,
the Bernoulli's law $\p_t\phi+\frac12|\nabla\phi|^2+h(\rho)=0$
holds with the enthalpy $h(\rho)$ satisfying $h'(\rho)=\frac{p'(\rho)}\rho$ and $h(\bar\rho)=0$. Therefore, for smooth
irrotational flows, $\phi$ fulfills the following second order quasilinear wave equation
\begin{equation}\label{potentia1}
\Box\phi=2\p_t\phi\Delta\phi-2\sum_{k=1}^3\p_k\phi\p_{tk}^2\phi
+|\nabla\phi|^2\Delta\phi-\sum_{i,j=1}^3\p_i\phi\p_j\phi\p_{ij}^2\phi,
\quad\Delta=\p_1^2+\p_2^2+\p_3^2,
\end{equation}
with the initial data $(\phi_{(0)}, \phi_{(1)})$ as
\begin{equation*}
\begin{split}
&\phi_{(0)}(x_1,x_2,y):=\int_{-\infty}^{x_1}v_1^0(s,x_2,y)ds,\\
&\phi_{(1)}=-\frac12(|v_1^0|^2+|v_2^0|^2+|v_3^0|^2)-h(\bar\rho+\rho_0).
\end{split}
\end{equation*}
One easily checks that the quadratic nonlinearity in \eqref{potentia1} can be written as
\begin{equation}\label{Tao;potentia1ofQ0}
2\p_t\phi\Delta\phi-2\sum_{k=1}^3\p_k\phi\p_{tk}^2\phi=2Q_0(\p_t\phi,\phi)-2\p_t\phi\Box\phi.
\end{equation}
By substituting \eqref{Tao;potentia1ofQ0} into the equation \eqref{potentia1} and replacing $\Box \phi$ with \eqref{potentia1}, 
we can transform the equation \eqref{potentia1} into the following form:
\begin{equation}\label{potentia1-Final}
\begin{split}
\Box\phi&=2Q_0(\p_t\phi,\phi)-4(\p_t\phi)^2\Delta\phi+4\sum_{k=1}^3\p_t\phi\p_k\phi\p_{tk}^2\phi
+|\nabla\phi|^2\Delta\phi-\sum_{i,j=1}^3\p_i\phi\p_j\phi\p_{ij}^2\phi\\
&\quad-2\p_t\phi\Delta\phi|\nabla\phi|^2+2\sum_{i,j=1}^3\p_t\phi\p_i\phi\p_j\phi\p_{ij}^2\phi.
\end{split}
\end{equation}
It should be emphasized that the quartic term $2\sum_{i,j=1}^3\p_t\phi\p_i\phi\p_j\phi\p_{ij}^2\phi$
in the nonlinearity of \eqref{potentia1-Final} does not cause any additional difficulties. 
In addition, the cubic nonlinearities of \eqref{potentia1-Final} satisfy the partial null condition \eqref{part:null:condtion}.
This means that the 3D irrotational potential flow equation \eqref{potentia1} is a special case of the nonlinear wave
system in problem \eqref{QWE}.

\vskip 0.2 true cm
(2) {\bf 3D relativistic membrane equation}
\begin{equation}\label{HCC-X}
\p_t\bigg(\ds\f{\p_t\phi}{\sqrt{1-(\p_t\phi)^2+|\nabla\phi|^2}}\bigg)
-div\bigg(\ds\f{\nabla\phi}{\sqrt{1-(\p_t\phi)^2+|\nabla\phi|^2}}\bigg)=0,
\end{equation}
which corresponds to the Euler-Lagrange equation of the area functional
$\int_{\Bbb R\times\Bbb R^2}\int_{\Bbb T}\sqrt{1+|\nabla \phi|^2-(\p_t\phi)^2}\\dtdxdy$
for the embedding of $(t,x,y)\to (t,x,y,\phi(t,x,y))$ in the Minkowski space-time.
Here $x=(x_1,x_2)$, $\nabla=(\p_{x_1}, \p_{x_2}, \p_{y})$ and $\phi$ is
$2\pi$-periodic with respect to the variable $y$.

It is easy to know that \eqref{HCC-X}
is equivalent to the following nonlinear equation for the $C^2$ solution $\phi$
\begin{equation}\label{HCC-Y}
\Box\phi=(|\p_t\phi|^2-|\nabla\phi|^2)(\p_t^2\phi-\Delta\phi)
-(\p_t\phi)^2\p_t^2\phi+2\p_t\phi\nabla\phi\cdot\p_t\nabla\phi-\nabla\phi\cdot\nabla(|\nabla\phi|^2).
\end{equation}
Then \eqref{HCC-X} is a special case of
the nonlinear system in problem \eqref{QWE} since the related cubic nonlinearity in \eqref{HCC-Y} satisfies
the partial null condition  \eqref{part:null:condtion}.

\vskip 0.2 true cm
(3) {\bf Some 3D quasilinear wave equations which come from the Lagrangian functionals as perturbations of the
Lagrangian densities of linear waves}

\vskip 0.1 true cm
\begin{equation}\label{Son-00}
\Box\phi=-\p_t\bigg(\big(-\ds\f12(\p_t\phi)^2+\f12|\nabla\phi|^2\big)^{k}\p_t\phi\bigg)
+div\bigg(\big(-\ds\f12(\p_t\phi)^2+\f12|\nabla\phi|^2\big)^{k}\nabla\phi\bigg), \quad k\in\Bbb N,
\end{equation}
which is the Euler-Lagrangian equation of
$L(\phi)=-\ds\f12(\p_t\phi)^2+\f12|\nabla\phi|^2
+\f{1}{k+1}\bigg(-\ds\f12(\p_t\phi)^2+\f12|\nabla\phi|^2\bigg)^{k+1}$
with $\nabla=(\p_{x_1}, \p_{x_2}, \p_{y})$.
Note that $L(\phi)$ can be thought as a perturbation of the Lagrangian density of linear wave.

In addition,
\begin{equation*}\label{Son-1-0}
(1+(\p_t\phi)^{2k+1})\p_t^2\phi-\Delta\phi+\p_t\bigg(\big(-\ds\f12(\p_t\phi)^2+\f12|\nabla\phi|^2\big)^{k}\p_t\phi\bigg)
-div\bigg(\big(-\ds\f12(\p_t\phi)^2+\f12|\nabla\phi|^2\big)^{k}\nabla\phi\bigg)=0,
\end{equation*}
namely,
\begin{equation}\label{Son-1-0}
\Box\phi=-\p_t\bigg(\big(-\ds\f12(\p_t\phi)^2+\f12|\nabla\phi|^2\big)^{k}\p_t\phi\bigg)
+div\bigg(\big(-\ds\f12(\p_t\phi)^2+\f12|\nabla\phi|^2\big)^{k}\nabla\phi\bigg)-(\p_t\phi)^{2k+1}\p_t^2\phi,
\end{equation}
which corresponds to the Euler-Lagrangian equation of the functional
\begin{equation*}\label{Son-0}
L(\phi)=-\ds\f12(\p_t\phi)^2+\f12|\nabla\phi|^2
+\f{1}{k+1}\bigg(-\ds\f12(\p_t\phi)^2+\f12|\nabla\phi|^2\bigg)^{k+1}-\f{(\p_t\phi)^{2k+3}}{(2k+3)(2k+2)}\quad
(k\in\Bbb N).
\end{equation*}
We see  that both \eqref{Son-00} and \eqref{Son-1-0} are the special cases of
the nonlinear wave system in problem \eqref{QWE}.

\vskip 0.2 true cm
(4) {\bf 3D wave maps system}

\vskip 0.1 true cm
Let ($\mathcal{M}$, $g$) be a $m$-dimensional compact Riemannian manifold without
boundary. A wave map is a continuous function from the Minkowski space $\R^{1+3}$ into $\mathcal{M}$:
\begin{equation*}
\phi=(\phi^1, ..., \phi^m): \R\times\R^3\longrightarrow\mathcal{M},
\end{equation*}
which corresponds to a critical point of the following functional
\begin{equation*}
F(\phi)=\int_{\R^{1+3}}\langle\eta^{\alpha\beta}\p_\beta\phi,\p_\alpha\phi\rangle_{g}dtdx.
\end{equation*}
Here $\eta$ is the standard Minkowski metric on $\R^{1+3}$, which is represented by the matrix $\mathrm{diag}(-1,1,1,1)$ in
rectangular coordinates. When the local coordinates on $\mathcal{M}$ are introduced, the equations of $\phi^i~(1\leq i\leq m)$ can be written as
\begin{equation}\label{Tao;wavemap}
\Box\phi^i=\sum_{j,k=1}^3\Gamma^i_{jk}(\phi)\Big(\p_t\phi^j\p_t\phi^k-\sum_{l=1}^3\p_{x_l}\phi^j\p_{x_l}\phi^k\Big),~~~1\leq i\leq m,
\end{equation}
where $\Gamma^i_{jk}(\phi)$'s are the Christoffel symbols of the metric $g$. For more detailed derivation of the wave maps system \eqref{Tao;wavemap}, one can see  2.1.2 of Chapter 2 in \cite{Shatah98}. When we consider the small data solution
of \eqref{Tao;wavemap}, by carrying out the Taylor expansion on the
function $\Gamma^i_{jk}(\phi)$ under Riemann normal coordinates, the long time existence of \eqref{Tao;wavemap} can be
attributed to the following system of semilinear wave equations
\begin{equation}\label{Tao;wavemap22}
\Box\phi^i=\sum_{j,k=1}^3\sum_{l=1}^3C^i_{jkl}Q_0(\phi^j,\phi^k)\phi^l,~~~1\leq i\leq m.
\end{equation}
It is easy to know that the 3D wave maps system \eqref{Tao;wavemap} or \eqref{Tao;wavemap22} is a special case of
the nonlinear wave system in problem \eqref{QWE} when $\phi(t,x_1,x_2,y)$ is $2\pi$-periodic with respect to $y$.

We now return to the problem \eqref{QWE}. Set
\begin{equation}\label{initial:data}
\begin{split}
\ve&:=\sum_{i+j\le2N+1}\|\w{x}^{4+i+j}\nabla_x^i\p_y^ju_{(0)}\|_{L^2_{x,y}}
+\sum_{i+j\le2N}\|\w{x}^{4+i+j}\nabla_x^i\p_y^ju_{(1)}\|_{L^2_{x,y}},
\end{split}
\end{equation}
where $N\ge 17$, $\nabla_x=(\p_1,\p_2)$, $\|\cdot\|_{L^2_{x,y}}:=\|\cdot\|_{L^2(\R^2\times\T)}$ and $\w{x}:=\sqrt{1+|x|^2}$.

\vskip 0.3 true cm

The main results of this paper are:
\begin{theorem}\label{thm1}
Suppose that \eqref{sym:condition} and \eqref{part:null:condtion} hold.
Then there is an $\ve_0>0$ such that when $\ve\le\ve_0$, problem \eqref{QWE}
admits a global solution $u\in\bigcap\limits_{j=0}^{2N+1}C^{j}([0,\infty), H^{2N+1-j}(\R^2\times\T))$, which
satisfies that for a generic constant $C>0$,
\begin{equation}\label{thm1:decay}
\begin{split}
&|u(t,x,y)|\le C\ve\w{t+|x|}^{-1/2}\w{t-|x|}^{-0.4},\quad|\p u(t,x,y)|\le C\ve\w{t+|x|}^{-1/2}\w{t-|x|}^{-1/2},\\
&|\p_y u(t,x,y)|\le C\ve\w{t+|x|}^{-1}.
\end{split}
\end{equation}
Furthermore, $u$ scatters to a free solution: there exists $(u_{(0)}^\infty,u_{(1)}^\infty)\in\dot H^1(\R^2\times\T)\times L^2(\R^2\times\T)$ such that
\begin{equation}\label{thm1:scater}
\|\p(u(t,\cdot)-u^\infty(t,\cdot))\|_{L^2(\R^2\times\T)}\le C\ve^2(1+t)^{-1/2}\rightarrow 0~~\mathrm{as}~~t\rightarrow+\infty,
\end{equation}
where $u^\infty$ is the solution of the homogeneous linear wave equation $\Box u^\infty=0$ on $\R^2\times\T$ with some initial data $(u_{(0)}^\infty,u_{(1)}^\infty)(x,y)$ being $2\pi$-periodic with respect to $y$.
\end{theorem}

\begin{remark} {\it By analogous proof, it can be known that
Theorem \ref{thm1} still holds for the fully nonlinear case $\sum_{j,k=1}^m\sum_{|a|,|b|\le1}C^{ab}_{ijk}Q_0(\p^au^j,\p^bu^k)
+\sum_{j,k,l=1}^m\sum_{|a|,|b|\le1}C^{ab}_{ijkl}Q_0(\p^au^j,\p^bu^k)u^l$ in problem \eqref{QWE}
when the related symmetric conditions like \eqref{sym:condition} are imposed.}
\end{remark}

\begin{remark}
{\it The $L^2(\R^2\times\T)$ norm in \eqref{thm1:scater} can be improved to the $H^N(\R^2\times\T)$ norm.
In addition, it is noted that the nonlinear wave equation on $\R^2\times\T$ may be decomposed into a system of wave equations coupled with an infinite number of Klein-Gordon equations on $\R^2$ as in (1.18) of \cite{HTY24}. Indeed, by performing the Fourier transformation with respect to the periodic variable $y$ for the equation in \eqref{QWE}, one can obtain an infinite number of coupled nonlinear equations on $\R^2$ as follows
\begin{equation}\label{reduction to WKG}
(\p_t^2-\Delta_x+|n|^2)u^i_n(t,x)=(F^i(u, \p u,\p^2u))_n(t,x),\quad x\in \R^2, n\in\Z, i=1,\cdots,m,
\end{equation}
where
\begin{equation}\label{reduction to WKG-00}
u^i_n(t,x):=\frac{1}{2\pi}\int_{\T}e^{-ny\sqrt{-1}}u^i(t,x,y)dy,\quad n\in \Z.
\end{equation}
From \eqref{reduction to WKG}, one immediately sees that the zero-modes $u^i_0$ are the solutions of wave equations
and other non-zero modes $u^i_n$ ($n\in \Z \setminus\{0\}$) solve the Klein-Gordon equations with mass $|n|$.
By the proof on \eqref{thm1:scater},
we can conclude that all the zero-mode solutions and non-zero mode solutions exhibit linear scattering in
energy spaces.}
\end{remark}

\begin{remark}
{\it It is pointed out that for the global existence of low regularity solutions to the $n$-dimensional wave maps system \eqref{Tao;wavemap22} in  $\R^{1+n}$
($n\ge 2$), there have been remarkable results. For examples, when the initial data $(\phi, \p_t\phi)(0,x)$ is small in the homogeneous Besov space $\dot{B}^{2,1}_{n/2}(\Bbb R^n)\times\dot{B}^{2,1}_{n/2-1}(\Bbb R^n)$, the author in \cite{Tataru98,Tataru01} has proved the global small data
solution of \eqref{Tao;wavemap22} in the space $C([0, \infty), \dot{B}^{2,1}_{n/2}(\Bbb R^n))\cap C^1([0, \infty), \dot{B}^{2,1}_{n/2-1}(\Bbb R^n))$;
when $(\phi, \p_t\phi)(0,x)$ is small in the critical Sobolev space $\dot{H}^{n/2}(\Bbb R^n)\times\dot{H}^{n/2-1}(\Bbb R^n)$,
 the global small data solution of \eqref{Tao;wavemap22} in the space $C([0, \infty), \dot{H}^{n/2}(\Bbb R^n))
 \cap C^1([0, \infty), \dot{H}^{n/2-1}(\Bbb R^n))$
 was established in \cite{TaoIMRN,TaoCMP}. In addition, when $(\phi, \p_t\phi)(0,x)$ is supported compactly or decays fast
 for the space variable at infinity, and $(\phi, \p_t\phi)(0,x)$ is small in
 the Sobolev space $H^{k+1}(\Bbb R^n)\times H^{k}(\Bbb R^n)$ with $k\in\Bbb N$ being suitably large,
 the global existence and long time decay rate of small data smooth solution $\phi$
 to \eqref{Tao;wavemap22} are derived in \cite{Wong}.
 In the present paper, for the smooth small initial data without compact support,
 we can show the global solution of problem \eqref{QWE}  in $\Bbb R^2\times\Bbb T$ rather than in
 the whole space $\Bbb R^3$ or $\Bbb R^2$.}
\end{remark}

\begin{remark}
{\it For the 3D problem in \eqref{QWE}, if the initial data $(u_{(0)}, u_{(1)})\in C_0^{\infty}(\Bbb R^3)$ are sufficiently small
in the Sobolev space $H^{k+1}(\Bbb R^3)\times H^{k}(\Bbb R^3)$ with $k\in\Bbb N$ being suitably large,
then the global existence  of smooth small data solution $u$ can be obtained by collecting the results in
\cite{Christodoulou86,DLY,Hormander97book,Klainerman85,Klainerman,Wong}.
For the 2D problem in \eqref{QWE}, when the initial data $(u_{(0)}, u_{(1)})\in C_0^{\infty}(\Bbb R^2)$ are sufficiently small,
the global solution $u$ may be established when we apply the related  results or methods in \cite{Alinhac01a,Alinhac01b,LM,Wong}.
This motivates us to study the global classical solutions
to  problem \eqref{QWE} on $\R^2\times\T$.}
\end{remark}

\begin{remark}
{\it For  the 4D fully nonlinear scalar wave equation
with quadratic nonlinearity on $\R^3\times\T$
\begin{equation}\label{QWE-00}
\left\{
\begin{aligned}
&\Box u=\mathcal{Q}(\p u,\p^2u),\qquad(t,x,y)\in(0,\infty)\times\R^3\times\T,\\
&(u,\p_tu)(0,x,y)=(u_{(0)},u_{(1)})(x,y)
\end{aligned}
\right.
\end{equation}
or the 3D fully nonlinear scalar wave equation
with cubic nonlinearity on $\R^2\times\T$
\begin{equation}\label{QWE-cubic-00}
\left\{
\begin{aligned}
&\Box u=\mathcal{C}(\p u,\p^2 u),\qquad(t,x,y)\in(0,\infty)\times\R^2\times\T,\\
&(u,\p_tu)(0,x,y)=(u_{(0)},u_{(1)})(x,y),
\end{aligned}
\right.
\end{equation}
where $(u_{(0)},u_{(1)})(x,y)$ are $2\pi$-periodic with respect to the variable $y$,
under the partial null conditions on $\mathcal{Q}(\p u,\p^2u)$ and $\mathcal{C}(\p u,\p^2 u)$, we have shown the
global existence of smooth small data solutions to \eqref{QWE-00} and \eqref{QWE-cubic-00} in \cite{HTY24}.
However, the arguments in \cite{HTY24} are not suitable for the  quadratic nonlinearity $\mathcal{C}(\p u,\p^2 u)$ in \eqref{QWE-cubic-00}
due to the slow time decay rate of the solution to 2D linear wave equation. In the paper, we not only extend the global existence result
in \cite{HTY24} to some class of quadratic nonlinearities
(including the solution itself) on $\R^2\times\T$ but also to the wave equation system on $\R^2\times\T$.}
\end{remark}

\begin{remark}
{\it There are a lot of results on the coupled wave equations and Klein-Gordon equations with small initial data.
For examples, the author in \cite{Georgiev-90} introduced the strong null condition for the nonlinearities
of such forms $Q_{\alpha\beta}(u,v):=\p_\alpha u\p _\beta v-\p_\alpha v\p _\beta u$ ($\alpha,\beta\in \{0,1,2,3\}$)
and further established the global well-posedness
of small data solutions for
the coupled system of wave and Klen-Gordon equations in $\R^{1+3}$. Subsequently, the author in \cite{Katayama12}
proved the corresponding global existence of solutions under a weaker condition than the strong null condition proposed in \cite{Georgiev-90}.
In addition, the global small solution  has been investigated
for the following wave-Klein-Gordon system on $\R^{1+3}$:
\begin{equation}\label{Simplied-WKG-eq}
\begin{cases}
&-\Box u=A^{\alpha\beta}\p_\alpha v \p_\beta v+Dv^2,\\
&(-\Box +1)v=u B^{\alpha\beta}\p_\alpha\p_\beta v,
\end{cases}
\end{equation}
one can be referred to \cite{Q-Wang-2015,IP19,LeF-Ma-16,Q-Wang-2020}.
There are also many results on the global well-posedness of small data solutions to some coupled wave and Klein-Gordon equations in
lower space dimensions
(see \cite{Dong-JFA,Yue-M-6, Ifrim-Stingo,Stingo-Memoirs}).}
\end{remark}

\begin{remark}
Consider  the quadratic nonlinear wave equation
on $\R^2\times\T$
\begin{equation}\label{QWE-cubic}
\left\{
\begin{aligned}
&\Box u=Q(\p u,\p^2 u),\qquad(t,x,y)\in(0,\infty)\times\R^2\times\T,\\
&(u,\p_tu)(0,x,y)=(u_{(0)},u_{(1)})(x,y)
\end{aligned}
\right.
\end{equation}
with $Q(\p u,\p^2u)=\sum_{\alpha,\beta,\gamma,\delta=0}^3
F^{\alpha\beta\gamma\delta}\p^2_{\alpha\beta}u\p^2_{\gamma\delta}u
+\sum_{\alpha,\beta,\gamma=0}^3Q^{\alpha\beta\gamma}\p^2_{\alpha\beta}u\p_\gamma u
+\sum_{\alpha,\beta=0}^3S^{\alpha\beta}\p_\alpha u\p_\beta u$,
where $F^{\alpha\beta\gamma\delta},Q^{\alpha\beta\gamma},S^{\alpha\beta}$ are constants satisfying the symmetric conditions $F^{\alpha\beta\gamma\delta}=F^{\beta\alpha\gamma\delta}=F^{\alpha\beta\delta\gamma}=F^{\gamma\delta\alpha\beta}$, $Q^{\alpha\beta\gamma}=Q^{\beta\alpha\gamma}$
and $S^{\alpha\beta}=S^{\beta\alpha}$. Moreover, $Q(\p u,\p^2 u)$ satisfies the partial full condition
\eqref{part:null:condtion}.
It should be emphasized that for the general $Q(\p u,\p^2 u)$, we have not solved the global existence of
small data solution $u$ to \eqref{QWE-cubic} in $\R^2\times\T$
unless $Q(\p u,\p^2 u)$ admits a special structure as in \eqref{QWE}.
\end{remark}

The studies on the nonlinear wave equations (systems) on product spaces are motivated from
the Kaluza-Klein theory (see \cite{Kalu, Klein, Witten}) and the propagation of waves
along infinite homogeneous waveguides (see \cite{Ett-15,PLR-03,MSS-05}). On the other hand,
the partial periodic solutions to the nonlinear wave equations (systems) also make sense in physics.
The Kaluza-Klein theory
is closely related to the supersymmetric compactification of the space-time on the product manifold
\begin{equation}\label{Manifold}
\mathcal{M}^{1+n+d}=\R^{1+n}\times K,
\end{equation}
which admits a metric $\widehat{g}=\eta_{\R^{1,n}}+k$ that is globally hyperbolic and serves as a
solution to the $(1+n+d)$-dimensional vacuum Einstein equations, here $(K,k)$ represents a compact,
Ricci-flat $d$-dimensional Riemannian manifold.
Recently, the authors in \cite{HS-21} investigate a system of quasilinear wave equations with nonlinearities
being  some combinations of quadratic null forms on $\R^3\times\T$,
which can be seen as a toy model of Einstein equations with additional compact dimensions.
It is shown in \cite{HS-21} that the  global smooth solution exists for the small and regular initial data with polynomial
decays at infinity of $\Bbb R^3$. The authors in \cite{HSW2023} demonstrate the classical global stability of the flat
Kaluza-Klein space-time $\R^{1+3}\times \mathbb{S}^1$ under small perturbations.
In \cite{HTY24}, we establish the global or almost global small data solutions of 4D fully nonlinear wave
equations on $\R^{3}\times\T$ when the general partial null
condition is fulfilled or not, meanwhile, the global or almost global small data solutions
are also obtained for the 3D fully nonlinear wave equations with cubic nonlinearities on $\R^{2}\times\T$.
However, it is unknown in \cite{HTY24} whether the global small data solution exists when the
general quadratic  nonlinearities  of 3D nonlinear wave equations on $\R^{2}\times\T$ satisfy the
corresponding partial null conditions. In the present paper, we will focus on this unresolved problem
for a class of special nonlinearities in \eqref{QWE}.

We now give the comments on the proof of Theorem \ref{thm1}. Note that
\begin{equation}\label{IntroApp:A1}
\Box_{t,x,y}(fg)=g\Box_{t,x,y}f+f\Box_{t,x,y}g+2Q_0(f,g).
\end{equation}
This equality has been extensively utilized in \cite{TaoCMP} and \cite{Tataru01} to investigate the global existence of wave
maps problems $\Box u^i=\sum_{j,k,l=1}^mC_{ijkl}Q_0(u^j,u^k)u^l$ (this kind of nonlinearity also appears
in the second term on the right hand side of the first line of \eqref{QWE}).
Set $\ds\tilde V^i:=u^i-\frac12\sum_{j,k=1}^m\sum_{|a|+|b|\le1}C^{ab}_{ijk}\p^au^j\p^bu^k$. Then it follows from
\eqref{IntroApp:A1} and direct computation that the equation in \eqref{QWE} can be changed into
\begin{equation}\label{QWE2}
\begin{split}
\Box\tilde V^i&=\sum_{j,k,l=1}^m\sum_{|a|,|b|\le2}C^{ab}_{ijkl}\underbrace{Q_0(\p^au^j,\p^bu^k)u^l}_{\color{red}bad~term}
+\sum_{j,k,l=1}^m\sum_{\alpha,\beta,\mu,\nu=0}^3\tilde Q_{ijkl}^{\alpha\beta\mu\nu}\p^2_{\alpha\beta}u^j\p_{\mu}u^k\p_{\nu}u^l\\
&\quad+\sum_{j,k,l=1}^m\sum_{\alpha,\beta,\mu=0}^3\tilde Q_{ijkl}^{\alpha\beta\mu}\p_{\alpha}u^j\p_{\beta}u^k\p_{\mu}u^l+G_4^i(\p^{\le3}u),
\end{split}
\end{equation}
where  the definition of good error term $G_4^i(\p^{\le3}u)$ is given in \eqref{App:A3} of Appendix and $\tilde Q_{ijkl}^{\alpha\beta\mu\nu}$,
$\tilde Q_{ijkl}^{\alpha\beta\mu}$ are some suitable constants satisfying the related partial null conditions (see \eqref{App:A4}
of Appendix).
Although the right hand side of \eqref{QWE2} contains at least cubic nonlinearities on $\R^2\times\T$,
the cubic term $Q_0(\p^au^j,\p^bu^k)u^l$ is of the wave map type and is still hard to be handled for us
($Q_0(\p^au^j,\p^bu^k)u^l$ is also called ``bad term" in \eqref{QWE2}).
More specifically, the $L^\infty$ decay estimate of 
$Q_0(P_{\neq0}\p^au^j,P_{\neq0}\p^bu^k)P_{=0}u^l$ contained in $Q_0(\p^au^j,\p^bu^k)u^l$ cannot be well controlled
and further we cannot establish the crucial time-space decay estimate of $P_{=0}u$ directly (as in Lemma \ref{lem:pw0:improv} below),
where $P_{=0}f(t,x):=\frac{1}{2\pi}\int_{\T}f(t,x,y)dy$ and $P_{\neq0}:={\rm Id}-P_{=0}$.
To overcome this essential difficulty, we look for another new transformation as follows
\begin{equation}\label{normal:form}
\begin{split}
V^i&:=\tilde V^i-\frac12\sum_{j,k,l=1}^m\sum_{|a|,|b|\le2}C^{ab}_{ijkl}P_{\neq0}\p^au^jP_{\neq0}\p^bu^kP_{=0}u^l.
\end{split}
\end{equation}
In this case, \eqref{QWE2} is reformulated as
\begin{equation}\label{QWE3}
\begin{split}
&\Box V^i=\cC_1^i+\cC_2^i+G^i,\\
\end{split}
\end{equation}
where
\begin{equation}\label{QWE3-1}
\begin{split}
&\cC_1^i:=\sum_{j,k,l=1}^m\sum_{\alpha,\beta,\mu,\nu=0}^3\tilde Q_{ijkl}^{\alpha\beta\mu\nu}\p^2_{\alpha\beta}u^j\p_{\mu}u^k\p_{\nu}u^l
+\sum_{j,k,l=1}^m\sum_{\alpha,\beta,\mu=0}^3\tilde Q_{ijkl}^{\alpha\beta\mu}\p_{\alpha}u^j\p_{\beta}u^k\p_{\mu}u^l,\\
&\cC_2^i:=\sum_{j,k,l=1}^m\sum_{|a|,|b|\le2}C^{ab}_{ijkl}\{P_{=0}u^l[Q_0(P_{\neq0}\p^au^j,P_{=0}\p^bu^k)+Q_0(P_{=0}\p^au^j,P_{\neq0}\p^bu^k)\\
&\qquad+Q_0(P_{=0}\p^au^j,P_{=0}\p^bu^k)]+Q_0(\p^au^j,\p^bu^k)P_{\neq0}u^l\\
&\qquad-Q_0(P_{\neq0}\p^au^j,P_{=0}u^l)P_{\neq0}\p^bu^k-Q_0(P_{=0}u^l,P_{\neq0}\p^bu^k)P_{\neq0}\p^au^j\},\\
&G^i:=G_4^i(\p^{\le3}u)+\tilde G^i(\p^{\le4}u),
\end{split}
\end{equation}
and the definition of good error term $G^i$ is given in \eqref{App:A5} of Appendix.
Although the expressions of $\cC_1^i$, $\cC_2^i$ and $G^i$ seem to be complicated,
they actually have efficient forms which can be treated well by direct analysis in the lower
order energy estimates of $u$.

We next illustrate some other difficulties in the proof of Theorem \ref{thm1} and outline how to overcome them.

\vskip 0.2 true cm

{\bf $\bullet$ Invalidity of  the scaling vector $S=t\partial_t+\sum_{i=1}^{2}x_i\partial_{x_i}$}

\vskip 0.1 true cm

Note that $\R^2\times\T$ is not scaling invariant in the periodic direction $y$.
  Moreover, the action of $S+y\p_y$ does not preserve the periodic property, namely,
  $(S+y\p_y)f|_{y=2\pi}\neq(S+y\p_y)f|_{y=0}$ holds generally provided that $f$ is $2\pi$-periodic in $y$.
  On the other hand, it is well-known that the following Klainerman-Sobolev inequality plays a key role in
  deriving the time decay rate and further proving the global existence of small data smooth solutions to
  the $d$-dimensional nonlinear wave equations
     (see \cite{Alinhac01a,Alinhac01b,Hormander97book,Klainerman})
      \begin{equation}\label{KS:ineq}
      (1+|t-r|)^{\f12}(1+t)^{\f{d-1}{2}}|\p \phi(t,x)|\le C\ds\sum_{|I|\le [d/2]+2}\|\hat{Z}^I\p \phi(t,x)\|_{L_x^2(\Bbb R^d)},
      \end{equation}
     where $\hat{Z}\in\{\p, x_i\p_j-x_j\p_i, x_i\p_t+t\p_i, 1\le i,j\le d, t\partial_t+\sum_{k=1}^d x_k\p_k\}$ and $r=|x|=\sqrt{x_1^2+\cdot\cdot\cdot+x_d^2}$.
     Unfortunately, the scaling vector $S=t\partial_t+\sum_{i=1}^{2}x_i\partial_{x_i}$ does not commutate with the wave operator $\Box:=\p_t^2-\Delta_{x}-\p_y^2$. Even if we perform the Fourier transformation with respect to the periodic variable $y$
     for the nonlinear equations $\Box_{t,x,y}u^i(t,x,y)=F^i(u, \p u,\p^2u)(t,x,y)$ as in \eqref{reduction to WKG-00}, then
from \eqref{reduction to WKG}, one immediately sees that the zero-modes $u^i_0$ are the solutions of wave equations
and other non-zero modes $u^i_n$ ($n\in \Z \setminus\{0\}$) solve Klein-Gordon equations with mass $|n|$.
In this case,  $S=t\partial_t+\sum_{i=1}^{2}x_i\partial_{x_i}$ does not commutate with the resulting
 Klein-Gordon operators $\p_t^2-\Delta_x+|n|^2$ ($n\in \Z \setminus\{0\}$) yet. This fact prevents us
 from using the classical Klainerman-Sobolev inequality \eqref{KS:ineq} to show the global existence of \eqref{QWE}.
     To overcome this difficulty, our key ingredient is to establish the pointwise decay estimates for the solutions of \eqref{QWE}
    through deriving the space-time decay estimates of zero modes and non-zero modes by some kinds of
    $L^{\infty}-L^{\infty}$ and  $L^{\infty}-L^2$ estimates in the theory
     of linear wave equations. In this process, some basic techniques in wave equations
     (see \cite{KS96,Sideris,Lei16,Kubo19})  and in  Klein-Gordon equations (see \cite{Georgiev92})
     will be utilized and improved.

\vskip 0.2 true cm

{\bf $\bullet$ Loss of derivatives in the highest order energy estimates}

\vskip 0.1 true cm

Due to the transformation \eqref{normal:form} including the second order derivatives $\p^2u$, it follows from \eqref{QWE3}
and  \eqref{QWE3-1} that it is impossible to establish the highest order energy estimates of $u$ directly since the error
term $G^i$ in \eqref{QWE3-1}
contains $\p^4u$. To overcome this difficulty, we return to investigate the quadratic nonlinear problem \eqref{QWE} by
applying S.Alinhac's ghost weight technique and the lower order energy estimates of $u$ derived from \eqref{QWE3}. 
Nevertheless, in this procedure, we have to take
some new strategies to derive the precise a priori space-time decays of $P_{=0}u$
by making full use of the null condition structure in \eqref{QWE} and taking a lot of involved analyses,
which is different from the usual $L^2$ energy methods. More concretely,
we will derive
     \begin{equation}\label{hierarchy1}
     \begin{split}
     \sum_{|a|\le 2N-2}|P_{=0}\p Z^au(t,x,y)|&\le C\ve\w{t}^{0.01}\w{x}^{-1/2},\\
     \sum_{|a|\le 2N-5}|P_{=0}\p Z^au(t,x,y)|&\le C\ve\w{t}^{0.02}\w{x}^{-1/2}\w{t-|x|}^{-1/2},\\
     \sum_{|a|\le N+6}|P_{=0}\p Z^au(t,x,y)|&\le C\ve\w{x}^{-1/2}\w{t-|x|}^{-1.3}
     \end{split}
     \end{equation}
and
     \addtocounter{equation}{1}
     \begin{align}
     \sum_{|a|\le 2N-5}|P_{=0}Z^au(t,x,y)|&\le C\ve\w{t}^{0.02},\tag{\theequation a}\label{hierarchy2a}\\
     \sum_{|a|\le N+7}|P_{=0}Z^au(t,x,y)|&\le C\ve\w{t+|x|}^{-1/2}\w{t-|x|}^{-0.3},\tag{\theequation b}\label{hierarchy2b}\\
     \sum_{|a|\le N-1}|P_{=0}Z^au(t,x,y)|&\le C\ve\w{t+|x|}^{-1/2}\w{t-|x|}^{-0.4}.\tag{\theequation c}\label{hierarchy2c}
     \end{align}
  It should be emphasized that \eqref{hierarchy2b} and \eqref{hierarchy2c} need to be shown simultaneously
  by establishing the weighted  pointwise estimate on the solutions of 2D linear wave equations and cannot be
  treated separately (otherwise, it is difficult for us to close them).

When all the difficulties mentioned above are overcome and further the related estimates are established,
 the proof of Theorem \ref{thm1} can be completed by the continuity argument.

\vskip 0.2 true cm

\noindent \textbf{Notations:}
\begin{itemize}
  \item $\w{x}:=\sqrt{1+|x|^2}$.
  \item $\N_0:=\{0,1,2,\cdots\}$, $\Z:=\{0,\pm1,\pm2,\cdots\}$ and $\Z_*:=\Z\setminus\{0\}$.
  \item $\p_0:=\p_t$, $\p_1:=\p_{x_1}$, $\p_2:=\p_{x_2}$, $\p_3:=\p_y$, $\p_x:=\nabla_x=(\p_1,\p_2)$, $\p_{t,x}:=(\p_0,\p_1,\p_2)$, $\p:=(\p_t,\p_1,\p_2,\p_y)$, and $\Box_{t,x}:=\p_t^2-\Delta_x=\p_t^2-\p_1^2-\p_2^2$.
  \item For $|x|>0$, define $\bar\p_i:=\p_i+\frac{x_i}{|x|}\p_t$, $i=1,2$ and $\bar\p=(\bar\p_1,\bar\p_2)$.
 For $|x|=0$, $\bar\p$ is $\p_{x}$.
  \item On $\R^2\times\T$, we define: $L_i:=x_i\p_t+t\p_i,i=1,2$, $L:=(L_1,L_2)$, $\Omega:=\Omega_{12}:=x_1\p_2-x_2\p_1$, $\Gamma=\{\Gamma_1,\cdots,\Gamma_{6}\}:=\{\p_{t,x},L_1,L_2,\Omega_{12}\}$ and $Z=\{Z_1,\cdots,Z_{7}\}:=\{\Gamma,\p_y\}$.
  \item $\Gamma^a:=\Gamma_1^{a_1}\Gamma_2^{a_2}\cdots\Gamma_{6}^{a_{6}}$ for $a\in\N_0^{6}$
      and $Z^a:=Z_1^{a_1}Z_2^{a_2}\cdots Z_{7}^{a_{7}}$ for $a\in\N_0^7$.
  \item $\|\cdot\|_{L^p_x}:=\|\cdot\|_{L^p(\R^2)}$, $\|\cdot\|_{L^p_y}:=\|\cdot\|_{L^p(\T)}$ and $\|\cdot\|_{L^p_{x,y}}:=\|\cdot\|_{L^p(\R^2\times\T)}$.
  \item $P_{=0}f(t,x):=\frac{1}{2\pi}\int_{\T}f(t,x,y)dy$ and $P_{\neq0}:={\rm Id}-P_{=0}$.
  \item Auxiliary $k$-order energy of vector value function $V=(V^1, ..., V^m)$ on $\R^2\times\T$:
  \begin{equation*}
   \cX_k[V](t):=\sum_{i=1}^m\sum_{|a|\le k-1}\|\w{t-|x|}P_{=0}\p^2Z^aV^i\|_{L^2(\R^2)}.
  \end{equation*}
  \item $\ds f_n(t,x):=\frac{1}{2\pi}\int_{\T}e^{-ny\sqrt{-1}}f(t,x,y)dy$, $n\in\Z$, $x\in\R^2$.
  \item For $f,g\ge0$, $f\ls g$ means $f\le Cg$ for a generic constant $C>0$.
  \item $\ds|\p^{\le j}f|:=\Big(\sum_{0\le|a|\le j}|\p^af|^2\Big)^\frac12$.
  \item $\ds|\cP u|:=\Big(\sum_{i=1}^m|\cP u^i|^2\Big)^\frac12$ with $\cP\in\{\bar\p,Z\}$.
\end{itemize}

\vskip 0.1 true cm

The paper is organized as follows.
In Section 2, some preliminaries such as several basic lemmas and the bootstrap assumptions are given.
The pointwise estimates of the zero mode together with its higher and lower order derivatives will be established in Section 3.
In Section 4, the pointwise estimates of the non-zero modes are obtained.
With these pointwise estimates and the ghost weight technique, the required energy estimates can be achieved in Section 5.
In Section 6, we complete the proof of Theorem \ref{thm1}.
In addition, the detailed derivations of \eqref{QWE2} and \eqref{QWE3-1} are given in Appendix.

\section{Preliminaries and bootstrap assumptions}\label{sect2}
\subsection{Some basic lemmas}\label{sect2-1}

At first, we show some properties on the projection operators $P_{=0}$ and $P_{\neq0}$.
\begin{lemma}
For any real valued functions $f(t,x,y)$ and $g(t,x,y)$ on $\mathbb{R}^2\times\mathbb{T}$, it holds that
\begin{equation}\label{proj:property}
\begin{split}
&\int_{\T}P_{=0}fP_{\neq0}gdy=0,\quad
\|f\|^2_{L_y^2}=\|P_{=0}f\|^2_{L_y^2}+\|P_{\neq0}f\|^2_{L_y^2},\\
&P_{=0}Zf=ZP_{=0}f,\quad P_{\neq0}Zf=ZP_{\neq0}f,\quad Z=\{\p,L,\Omega\},\\
&P_{=0}\p_yf=\p_yP_{=0}f=0,\quad(\Gamma f)_n=\Gamma(f_n),\quad \Gamma=\{\p_{t,x},L,\Omega\},\\
&P_{=0}(fg)=P_{=0}fP_{=0}g+P_{=0}(P_{\neq0}fP_{\neq0}g),\\
&P_{\neq0}(fg)=P_{\neq0}(P_{=0}fP_{\neq0}g)+P_{\neq0}(P_{\neq0}fP_{=0}g)+P_{\neq0}(P_{\neq0}fP_{\neq0}g),\\
&|P_{=0}f|\ls\|f\|_{L_y^\infty},\quad|P_{\neq0}f|\ls\|f\|_{L_y^\infty}.
\end{split}
\end{equation}
\end{lemma}
\begin{proof}
These properties can be directly verified, we omit the details here.
\end{proof}

\begin{lemma}
For any real valued functions $f(t,x,y)$ and $g(t,x,y)$ on $\mathbb{R}^2\times\mathbb{T}$, one has
\begin{equation}\label{Q0:commutate}
Z^aQ_0(f,g)=\sum_{b+c\le a}C_{abc}Q_0(Z^bf,Z^cg),
\end{equation}
where $C_{abc}$ are some constants.
\end{lemma}
\begin{proof}
Note that
\begin{equation*}
\begin{split}
\p_\alpha Q_0(f,g)&=Q_0(\p_\alpha f,g)+Q_0(f,\p_\alpha g),\\
L_iQ_0(f,g)&=Q_0(L_if,g)+Q_0(f,L_ig),\\
\Omega Q_0(f,g)&=Q_0(\Omega f,g)+Q_0(f,\Omega g).
\end{split}
\end{equation*}
Then it follows from these identities and direct computation that \eqref{Q0:commutate} holds.
\end{proof}

The following two lemmas imply that the vector fields $Z^a$ commute with the wave operator $\Box$ and  the
resulting nonlinearities still fulfill the corresponding partial null conditions.
\begin{lemma}\label{lem:eqn:high}
Let $u(t,x,y)$ be a smooth solution of \eqref{QWE} and suppose that \eqref{part:null:condtion} holds.
Then for any multi-index $a$, $Z^au$ satisfies
\begin{equation*}\label{eqn:high}
\begin{split}
\Box Z^au^i=\sum_{\substack{j,k=1,\cdots,m\\ b+c\le a}}\bigg(\sum_{\alpha,\beta=0}^3Q_{ijk,abc}^{\alpha\beta}\p_{\alpha}Z^bu^j\p_{\beta}Z^cu^k
+\sum_{\alpha,\beta,\mu=0}^3Q_{ijk,abc}^{\alpha\beta\mu}\p^2_{\alpha\beta}Z^bu^j\p_{\mu}Z^cu^k\bigg)\\
\end{split}
\end{equation*}

\begin{equation}\label{eqn:high}
\begin{split}
+\sum_{\substack{j,k,l=1,\cdots,m\\ b+c+d\le a}}\bigg(C^{abcd}_{ijkl}Q_0(Z^bu^j,Z^cu^k)Z^du^l
+\sum_{\alpha,\beta,\mu,\nu=0}^3Q_{ijkl,abcd}^{\alpha\beta\mu\nu}\p^2_{\alpha\beta}Z^bu^j\p_{\mu}Z^cu^k\p_{\nu}Z^du^l\bigg),
\end{split}
\end{equation}
where $Q_{ijk,abc}^{\alpha\beta},Q_{ijk,abc}^{\alpha\beta\mu},C^{abcd}_{ijkl}$ and $Q_{ijkl,abcd}^{\alpha\beta\mu\nu}$ are constants,
$Q_{ijk,aa0}^{\alpha\beta\mu}=Q_{ijk}^{\alpha\beta\mu}$ and $Q_{ijkl,aa00}^{\alpha\beta\mu\nu}=Q_{ijkl}^{\alpha\beta\mu\nu}$.
Furthermore, for any $(\xi_0,\xi_1,\xi_2)\in\{\pm1\}\times\SS$,
\begin{equation}\label{null:high}
\begin{split}
\sum_{\alpha,\beta=0}^2Q_{ijk,abc}^{\alpha\beta}\xi_\alpha\xi_\beta\equiv0,\quad
\sum_{\alpha,\beta,\mu=0}^2Q_{ijk,abc}^{\alpha\beta\mu}\xi_\alpha\xi_\beta\xi_\mu\equiv0,\quad
\sum_{\alpha,\beta,\mu,\nu=0}^2Q_{ijkl,abcd}^{\alpha\beta\mu\nu}\xi_\alpha\xi_\beta\xi_\mu\xi_\nu\equiv0.
\end{split}
\end{equation}
\end{lemma}
\begin{proof}
\eqref{eqn:high} can be obtained by direct verification.
\eqref{null:high} comes from Lemma 6.6.5 of \cite{Hormander97book}, $\p_y\Gamma=\Gamma\p_y$ with $\Gamma=\{\p_{t,x},L,\Omega\}$
and direct computation.
\end{proof}

\begin{remark}
Here we point out that from \eqref{NL:rewrite} and Lemma \ref{Q0:commutate}, the quadratic nonlinearities
in \eqref{eqn:high} still admit the corresponding structure  of $Q_0$.
\end{remark}

\begin{lemma}
Suppose that $V$ is defined by \eqref{normal:form}.
Then for any multi-index $a$, $Z^aV$ satisfies
\begin{equation}\label{eqn:normal}
\begin{split}
&\Box Z^aV^i=Z^a\cC_1^i+Z^a\cC_2^i+Z^aG^i,\\
\end{split}
\end{equation}
where $\cC_1^i, \cC_2^i, G^i$ are defined in \eqref{QWE3-1}, and
\begin{equation}\label{eqn:normal-1}
\begin{split}
&Z^a\cC_1^i=\sum_{\substack{j,k,l=1,\cdots,m\\ b+c+d\le a}}
\sum_{\alpha,\beta,\mu,\nu=0}^3\tilde Q_{ijkl,abcd}^{\alpha\beta\mu\nu}\p^2_{\alpha\beta}Z^bu^j\p_{\mu}Z^cu^k\p_{\nu}Z^du^l\\
&\qquad\quad+\sum_{\substack{j,k,l=1,\cdots,m\\ b+c+d\le a}}
\sum_{\alpha,\beta,\mu=0}^3\tilde Q_{ijkl,abcd}^{\alpha\beta\mu}\p_{\alpha}Z^bu^j\p_{\beta}Z^cu^k\p_{\mu}Z^du^l,\\
&\sum_{\alpha,\beta,\mu,\nu=0}^2\tilde Q_{ijkl,abcd}^{\alpha\beta\mu\nu}\xi_\alpha\xi_\beta\xi_\mu\xi_\nu\equiv0,\quad
\sum_{\alpha,\beta,\mu=0}^2\tilde Q_{ijkl,abcd}^{\alpha\beta\mu}\xi_\alpha\xi_\beta\xi_\mu\equiv0,\quad\forall(\xi_0,\xi_1,\xi_2)\in\{\pm1\}\times\SS.
\end{split}
\end{equation}
\end{lemma}
\begin{proof}
By \eqref{App:A4}-\eqref{App:A5} in Appendix and the analogous proof on Lemma \ref{lem:eqn:high},
\eqref{eqn:normal-1} can be shown.
\end{proof}

When the related null conditions hold, the following estimates will play important roles in deriving
the global classical solutions of problem \eqref{QWE}.
\begin{lemma}\label{lem:null}
Suppose that the constants $N_1^{\alpha\beta},N_2^{\alpha\beta\mu}$ and $N_3^{\alpha\beta\mu\nu}$ satisfy
that for any $\xi=(\xi_0,\xi_1,\xi_2)\in\{\pm1,\SS\}$,
\begin{equation*}
\sum_{\alpha,\beta=0}^2N_1^{\alpha\beta}\xi_\alpha\xi_\beta\equiv0,
\quad\sum_{\alpha,\beta,\mu=0}^2N_2^{\alpha\beta\mu}\xi_\alpha\xi_\beta\xi_\mu\equiv0,
\quad\sum_{\alpha,\beta,\mu,\nu=0}^2N_3^{\alpha\beta\mu\nu}\xi_\alpha\xi_\beta\xi_\mu\xi_\nu\equiv0.
\end{equation*}
Then for smooth functions $f,g,h$ and $w$ on $\R^2\times\T$, it holds that
\begin{equation}\label{null:structure}
\begin{split}
|Q_0(P_{=0}f,g)|&\ls|P_{=0}\bar\p f||\p g|+|P_{=0}\p f||\bar\p g|,\\
\Big|\sum_{\alpha,\beta=0}^2N_1^{\alpha\beta}\p_{\alpha}f\p_{\beta}g\Big|
&\ls|\bar\p f||\p g|+|\p f||\bar\p g|,\\
\Big|\sum_{\alpha,\beta,\mu=0}^2N_2^{\alpha\beta\mu}\p^2_{\alpha\beta}f\p_{\mu}g\Big|
&\ls|\bar\p\p f||\p g|+|\p^2 f||\bar\p g|,\\
\Big|\sum_{\alpha,\beta,\mu=0}^2N_2^{\alpha\beta\mu}\p_{\alpha}f\p_{\beta}g\p_{\mu}h\Big|
&\ls|\bar\p f||\p g||\p h|+|\p f||\bar\p g||\p h|+|\p f||\p g||\bar\p h|,\\
\Big|\sum_{\alpha,\beta,\mu,\nu=0}^2N_3^{\alpha\beta\mu\nu}
\p^2_{\alpha\beta}f\p_{\mu}g\p_{\nu}h\Big|&\ls|\bar\p\p f||\p g||\p h|+|\p^2f||\bar\p g||\p h|+|\p^2f||\p g||\bar\p h|,\\
\Big|\sum_{\alpha,\beta,\mu,\nu=0}^2N_3^{\alpha\beta\mu\nu}
\p_{\alpha}f\p_{\beta}g\p_{\mu}h\p_{\nu}w\Big|&\ls|\bar\p f||\p g||\p h||\p w|+|\p f||\bar\p g||\p h||\p w|\\
&\quad+|\p f||\p g||\bar\p h||\p w|+|\p f||\p g||\p h||\bar\p w|,
\end{split}
\end{equation}
where $\ds|\bar\p f|:=\Big(\sum_{i=1}^2|\bar\p_if|^2\Big)^{1/2}$.
\end{lemma}
\begin{proof}
Since the proof of \eqref{null:structure} is analogous to Section 9.1 of \cite{Alinhac:book} and \cite[Lem 2.2]{HouYin20jde},
it suffices to  prove the first inequality as an example.
From the definition of $Q_0$, one has
\begin{equation*}
Q_0(P_{=0}f,g)=\p_tP_{=0}f\p_tg-\sum_{j=1}^3\p_jP_{=0}f\p_jg=\p_tP_{=0}f\p_tg-\sum_{j=1}^2\p_jP_{=0}f\p_jg.
\end{equation*}
This, together with the definition of $\bar\p_i=\p_i+\frac{x_i}{|x|}\p_t,i=1,2$ and direct algebraic computation,
yields the first inequality of \eqref{null:structure}.
\end{proof}

Next, we introduce some inequalities including the Poincar\'{e} inequality and Hardy inequality.
\begin{lemma}
For any function $f(y)$, $y\in\T$, it holds that
\begin{equation}\label{Poincare:ineq}
\|P_{\neq0}f\|_{L^2(\T)}=\|f-\bar f\|_{L^2(\T)}\ls\|\p_yf\|_{L^2(\T)},\quad\bar f:=\frac{1}{2\pi}\int_{\T}f(y)dy.
\end{equation}
For any function $g(t,x)\in L^2_x(\Bbb R^2)$, $x\in\R^2$, it holds that
\begin{equation}\label{Hardy:ineq}
\Big\|\frac{g}{|x|\ln|x|}\Big\|_{L^2(\R^2)}\ls\|\nabla_xg\|_{L^2(\R^2)}.
\end{equation}
Furthermore, it also holds that for any fixed positive constant $\lambda$,
\begin{equation}\label{mod:Hardy:ineq}
\Big\|\frac{g\chi}{\w{|x|-t}^{1/2+\lambda}}\Big\|_{L^2(\R^2)}\ls\frac{\w{t}^{1/2}\ln(2+t)}{\lambda^{1/2}}\|\nabla_xg\|_{L^2(\R^2)},
\end{equation}
where $\chi=\chi(\frac{|x|}{\w{t}})$, $\chi(s)\in C^\infty(\R)$, $0\le\chi(s)\le1$ and
\begin{equation}\label{cutoff:def}
\chi(s)=\left\{
\begin{aligned}
&1,\qquad s\in[1/2,2],\\
&0,\qquad s\not\in[1/3,3].
\end{aligned}
\right.
\end{equation}
\end{lemma}
\begin{proof}
\eqref{Poincare:ineq} is the Poincar\'{e} inequality and \eqref{Hardy:ineq} is the Hardy inequality.
Set $q(s)=\int_{-\infty}^s\frac{dt}{(1+|t|^2)^{1/2+\lambda}}$. Then $|q(r-t)|\ls\lambda^{-1}$ holds.
Hence \eqref{mod:Hardy:ineq} can be achieved by the following computation
\begin{equation*}
\begin{split}
&\Big\|\frac{g\chi}{\w{|x|-t}^{1/2+\lambda}}\Big\|^2_{L^2(\R^2)}=\int_0^\infty\int_{\SS}\frac{g^2\chi^2}{(1+|r-t|^2)^{1/2+\lambda}}rd\omega dr
=\int_0^\infty\int_{\SS}g^2\chi^2rd\omega dq(r-t)\\
&\quad \ls\int_0^\infty\int_{\SS}g^2|q(r-t)||\p_r(\chi^2r)|d\omega dr+\int_0^\infty\int_{\SS}|q(r-t)||g\p_rg|\chi^2rd\omega dr\\
&\quad \ls\lambda^{-1}\Big(\w{t}^{-1}\int_{\R^2}g^2\chi dx+\|g\chi\|_{L^2(\R^2)}\|\nabla_xg\|_{L^2(\R^2)}\Big)\\
&\quad \ls\lambda^{-1}\w{t}\ln^2(2+t)\int_{\R^2}\frac{g^2}{|x|^2\ln^2|x|}dx
+\lambda^{-1}\w{t}\ln(2+t)\Big\|\frac{g}{|x|\ln|x|}\Big\|_{L^2(\R^2)}\|\nabla_xg\|_{L^2(\R^2)}\\
&\quad \ls\lambda^{-1}\w{t}\ln^2(2+t)\|\nabla_xg\|^2_{L^2(\R^2)}.
\end{split}
\end{equation*}
\end{proof}

\subsection{Bootstrap assumptions}\label{sect2-2}

Define the energy for problem \eqref{QWE}
\begin{equation*}
E_k[u](t):=\sum_{i=1}^m\sum_{|a|\le k}\|\p Z^au^i\|_{L^2(\R^2\times\T)}.
\end{equation*}
We make the following bootstrap assumptions:
\begin{equation}\label{BA1}
E_{2N}[u](t)\le\ve_1(1+t)^{\ve_2}
\end{equation}
and
\begin{align}
&\sum_{|a|\le N+7}|P_{=0}Z^au(t,x,y)|\le\ve_1\w{t+|x|}^{-1/2}\w{t-|x|}^{-0.3},\label{BA2}\\
&\sum_{|a|\le N+6}|P_{=0}\p Z^au(t,x,y)|\le\ve_1\w{x}^{-1/2}\w{t-|x|}^{-1.3},\label{BA3}\\
&\sum_{|a|\le2N-8}|P_{\neq0}Z^au(t,x,y)|\le\ve_1\w{t+|x|}^{\ve_2-1/2}(\w{t+|x|}^{-0.45}+\w{x}^{-1/2}\w{t-|x|}^{-0.4}),\label{BA4}\\
&\sum_{|a|\le N+2}|P_{\neq0}Z^au(t,x,y)|\le\ve_1\w{t+|x|}^{-1},\label{BA5}
\end{align}
where $\ve_1\in(\ve,1)$ will be determined later and $\ve_2=1/100$ .

By \eqref{proj:property}, \eqref{BA2}-\eqref{BA5}, one can derive the following results.
\begin{corollary}
Let $u(t,x,y)$ be the solution of \eqref{QWE} and suppose that \eqref{BA2}-\eqref{BA5} hold.
Then we have that for  $N\ge17$,
\begin{align}
&\sum_{|a|\le N+7}|Z^au|\ls\ve_1\w{t+|x|}^{-1/2}\w{t-|x|}^{-0.3},\label{BA:pw:full}\\
&\sum_{|a|\le N+6}|P_{=0}\bar\p Z^au(t,x,y)|\ls\ve_1\w{x}^{-1/2}\w{t+|x|}^{-1}\w{t-|x|}^{-0.3},\label{pw0:good:low}\\
&\sum_{|a|\le N+1}|P_{\neq0}\bar\p Z^au(t,x,y)|\ls\ve_1\w{t+|x|}^{-2}\w{t-|x|}.\label{pwKG:good}
\end{align}
\end{corollary}
\begin{proof}
\eqref{BA:pw:full} can be achieved by \eqref{proj:property}, \eqref{BA2} and \eqref{BA4}.

In the region of $|x|\le\w{t}/2$, $\w{t-|x|}\approx\w{t+|x|}$ holds.
In view of \eqref{BA3} and \eqref{BA5}, it suffices to prove \eqref{pw0:good:low} and \eqref{pwKG:good} in the region
of  $|x|\ge\w{t}/2$.
According to the definition of $\bar\p_i$, one obtains that for any function $f$,
\begin{equation}\label{good:ident}
|x|\bar\p_if=|x|(\frac{x_i}{|x|}\p_t+\p_i)f=L_if+(|x|-t)\p_if.
\end{equation}
Together with \eqref{BA2} and \eqref{BA3}, this yields
\begin{equation*}
\begin{split}
\w{x}\sum_{|a|\le N+6}|P_{=0}\bar\p Z^au(t,x,y)|
&\ls\sum_{|b|\le N+7}|P_{=0}Z^bu|+\w{t-|x|}\sum_{|a|\le N+6}|P_{=0}\p Z^au|\\
&\ls\ve_1\w{t+|x|}^{-1/2}\w{t-|x|}^{-0.3}+\ve_1\w{x}^{-1/2}\w{t-|x|}^{-0.3}.
\end{split}
\end{equation*}
Thus, the proof of \eqref{pw0:good:low} is completed.
Analogously, \eqref{pwKG:good} can be shown.
\end{proof}

\section{Pointwise estimates of the zero modes}\label{sect33}

This section aims to improve the bootstrap assumptions \eqref{BA2} and \eqref{BA3}.

\subsection{Higher order pointwise estimates of the zero modes}

\begin{lemma}
Let $u(t,x,y)$ be the smooth solution of \eqref{QWE}, then for any multi-index $a$ with $|a|\le2N-2$, it holds that
\begin{equation}\label{pw:high}
\begin{split}
\w{x}^{1/2}|P_{=0}\p Z^au(t,x,y)|&\ls E_{|a|+2}[u](t),\\
\w{x}^{1/2}|P_{\neq0}Z^au(t,x,y)|&\ls E_{|a|+2}[u](t).
\end{split}
\end{equation}
\end{lemma}
\begin{proof}
For any function $f(x)$ defined on $\R^2$, it can be obtained by the Sobolev embedding on $\R^2$ and \cite[Proposition 1]{Klainerman85} that
\begin{equation}\label{pw:high:pf1}
\w{x}^{1/2}|f(x)|\ls\sum_{|a|+|b|\le2}\|\nabla_x^a\Omega^bf\|_{L^2(\R^2)}.
\end{equation}
On the other hand, for any function $g(x,y)$ defined on $\R^2\times\T$, we can conclude from the Sobolev embedding theorem on $\T$ and the Poincar\'{e} inequality \eqref{Poincare:ineq} that
\begin{equation*}
\begin{split}
|P_{=0}g(x,y)|&\ls\|P_{=0}g(x,y)\|_{L^2(\T)}+\|\p_yP_{=0}g(x,y)\|_{L^2(\T)}\ls\|P_{=0}g(x,y)\|_{L^2(\T)},\\
|P_{\neq0}g(x,y)|&\ls\|P_{\neq0}g(x,y)\|_{L^2(\T)}+\|P_{\neq0}\p_yg(x,y)\|_{L^2(\T)}\ls\|P_{\neq0}\p_yg(x,y)\|_{L^2(\T)}.
\end{split}
\end{equation*}
This, together with \eqref{proj:property} and \eqref{pw:high:pf1}, yields \eqref{pw:high}.
\end{proof}

\begin{lemma}\label{lem:Sobolev}
Let $V$ be defined by \eqref{normal:form}, then for any multi-index $a$ with $|a|\le2N-5$, we have that
\begin{equation}\label{pw0:V1}
\begin{split}
&\w{x}^{1/2}\w{t-|x|}^{1/2}|P_{=0}\p Z^aV(t,x,y)|\ls E_{|a|+2}[V](t)+\cX_{|a|+2}[V](t),\\
&\cX_k[V](t):=\sum_{i=1}^m\sum_{|b|\le k-1}\|\w{t-|x|}P_{=0}\p^2Z^bV^i\|_{L^2(\R^2)}.
\end{split}
\end{equation}
\end{lemma}
\begin{proof}
See (3.12) of \cite{HTY24} for the proof.
\end{proof}
To apply Lemma \ref{lem:Sobolev}, it is required to derive the estimates of $E_{2N-3}[V](t)$ and $\cX_{2N-3}[V](t)$
on the right hand side of \eqref{pw0:V1}.
\begin{lemma}
Let $V$ be defined by \eqref{normal:form}, one has
\begin{equation}\label{energy:V}
E_{2N-3}[V](t)+\cX_{2N-3}[V](t)\ls\ve_1\w{t}^{2\ve_2}.
\end{equation}
\end{lemma}
\begin{proof}
First, we establish the estimate
\begin{equation}\label{energy:V:pf1}
\sum_{|a|\le2N-3}\|\w{t+|x|}\Box Z^aV(t,x,y)\|_{L^2(\R^2\times\T)}\ls\ve_1^3\w{t}^{\ve_2}\ln(2+t).
\end{equation}

Note that $\Box Z^aV^i=Z^a\cC_1^i+Z^a\cC_2^i+Z^aG^i$.

For $Z^a\cC_1^i$, there is at most one of the three factors in $\p^2_{\alpha\beta}Z^bu^j\p_{\mu}Z^cu^k\p_{\nu}Z^du^l$ and $\p_{\alpha}Z^bu^j\p_{\beta}Z^cu^k\p_{\mu}Z^du^l$ which does not fulfill \eqref{BA:pw:full}.
For this factor, its $L^2_{x,y}$ estimate will be utilized.
Then it follows from \eqref{BA1} and \eqref{BA:pw:full} that
\begin{equation}\label{energy:V:pf2}
\sum_{|a|\le2N-3}\|\w{t+|x|}Z^a\cC_1^i\|_{L^2(\R^2\times\T)}\ls\ve_1^2E_{2N}[u](t)\ls\ve_1^3\w{t}^{\ve_2}.
\end{equation}
For $Z^a\cC_2^i$, collecting \eqref{proj:property}, \eqref{Q0:commutate} and \eqref{null:structure} yields
\begin{equation}\label{energy:V:pf3}
\begin{split}
&|Z^a\cC_2^i|\ls\sum_{|b|+|c|+|d|\le|a|}\Big\{|P_{=0}Z^bu||P_{=0}\bar\p \p^{\le2}Z^cu|(|\p^{\le2}P_{=0}\p Z^du|+|\p^{\le3}P_{\neq0}Z^du|)\\
&+|P_{=0}Z^bu||\p^{\le2}P_{=0}\p Z^cu||P_{\neq0}\bar\p \p^{\le2}Z^du|\\
&+|P_{\neq0}Z^bu||\p^{\le2}\p Z^cu||\p^{\le2}\p Z^du|
+|P_{=0}\p Z^bu||\p^{\le3}P_{\neq0}Z^cu||\p^{\le3}P_{\neq0}Z^du|\Big\}.
\end{split}
\end{equation}
When $|b|\le N+2$, it follows from \eqref{BA2}, \eqref{BA3} and \eqref{BA5} that
\begin{equation}\label{energy:V:pf4}
\begin{split}
&\quad\|\w{t+|x|}|P_{=0}Z^bu||P_{=0}\bar\p \p^{\le2}Z^cu|(|\p^{\le2}P_{=0}\p Z^du|+|\p^{\le3}P_{\neq0}Z^du|)\|_{L^2(\R^2\times\T)}\\
&\quad+\|\w{t+|x|}|P_{=0}Z^bu||\p^{\le2}P_{=0}\p Z^cu||P_{\neq0}\bar\p \p^{\le2}Z^du|\|_{L^2(\R^2\times\T)}\\
&\quad+\|\w{t+|x|}|P_{\neq0}Z^bu||\p^{\le2}\p Z^cu||\p^{\le2}\p Z^du|\|_{L^2(\R^2\times\T)}\\
&\quad+\|\w{t+|x|}|P_{=0}\p Z^bu||\p^{\le3}P_{\neq0}Z^cu||\p^{\le3}P_{\neq0}Z^du|\|_{L^2(\R^2\times\T)}\\
&\ls\ve_1\|\w{t+|x|}^{1/2}|\p^{\le2}P_{=0}\p Z^cu|(|\p^{\le2}P_{=0}\p Z^du|+|\p^{\le3}P_{\neq0}Z^du|)\|_{L^2_{x,y}}\\
&\quad+\ve_1\||\p^{\le2}\p Z^cu||\p^{\le2}\p Z^du|\|_{L^2_{x,y}}+\ve_1\|\w{t+|x|}^{1/2}|\p^{\le3}P_{\neq0}Z^cu||\p^{\le3}P_{\neq0}Z^du|\|_{L^2_{x,y}}\\
&\ls\ve_1^2E_{2N}[u](t)\ls\ve_1^3\w{t}^{\ve_2}.
\end{split}
\end{equation}
When $|b|\ge N+3$, $|b|+|c|+|d|\le|a|\le2N-3$ ensures $|c|+|d|\le N-6$.
On the other hand, one can easily see that for any function $f$,
\begin{equation}\label{energy:V:pf5}
|Zf|\ls\w{t+|x|}|\p f|.
\end{equation}
For the terms in the first and second lines of \eqref{energy:V:pf3}, by \eqref{BA3}, \eqref{BA5}, \eqref{pw0:good:low}
and \eqref{pwKG:good},
we arrive at
\begin{equation*}
\begin{split}
&\quad|P_{=0}\bar\p \p^{\le2}Z^cu|(|\p^{\le2}P_{=0}\p Z^du|+|\p^{\le3}P_{\neq0}Z^du|)
+|\p^{\le2}P_{=0}\p Z^cu||P_{\neq0}\bar\p \p^{\le2}Z^du|\\
&\ls\ve_1^2\w{x}^{-1/2}\w{t+|x|}^{-1}(\w{x}^{-1/2}\w{t-|x|}^{-1.3}+\w{t+|x|}^{-1})\\
&\quad+\ve_1^2\w{x}^{-1/2}\w{t-|x|}^{-1.3}\w{t+|x|}^{-2}\w{t-|x|}\\
&\ls\ve_1^2\w{t+|x|}^{-2}.
\end{split}
\end{equation*}
This, together with \eqref{BA1} and \eqref{energy:V:pf5}, yields
\begin{equation}\label{energy:V:pf6}
\begin{split}
&\quad\|\w{t+|x|}|P_{=0}Z^bu||P_{=0}\bar\p \p^{\le2}Z^cu|(|\p^{\le2}P_{=0}\p Z^du|+|\p^{\le3}P_{\neq0}Z^du|)\|_{L^2(\R^2\times\T)}\\
&\quad+\|\w{t+|x|}|P_{=0}Z^bu||\p^{\le2}P_{=0}\p Z^cu||P_{\neq0}\bar\p \p^{\le2}Z^du|\|_{L^2(\R^2\times\T)}\\
&\ls\ve_1^2\sum_{|b'|\le|b|-1}\|P_{=0}\p Z^{b'}u\|_{L^2(\R^2\times\T)}\ls\ve_1^3\w{t}^{\ve_2}.
\end{split}
\end{equation}
Applying \eqref{Poincare:ineq}, \eqref{BA5} and \eqref{BA:pw:full} to the terms in the third line of \eqref{energy:V:pf3} leads to
\begin{equation}\label{energy:V:pf7}
\begin{split}
&\quad\|\w{t+|x|}|P_{\neq0}Z^bu||\p^{\le2}\p Z^cu||\p^{\le2}\p Z^du|\|_{L^2(\R^2\times\T)}\\
&\quad+\|\w{t+|x|}|P_{=0}\p Z^bu||\p^{\le3}P_{\neq0}Z^cu||\p^{\le3}P_{\neq0}Z^du|\|_{L^2(\R^2\times\T)}\\
&\ls\ve_1^2\|P_{\neq0}Z^bu\|_{L^2_{x,y}}+\ve_1^2\w{t}^{-1}\|P_{=0}\p Z^bu\|_{L^2_{x,y}}\\
&\ls\ve_1^2E_{2N}[u](t)\ls\ve_1^3\w{t}^{\ve_2}.
\end{split}
\end{equation}
Collecting \eqref{energy:V:pf2}-\eqref{energy:V:pf7} derives
\begin{equation}\label{energy:V:pf8}
\sum_{|a|\le2N-3}\|\w{t+|x|}Z^a(\cC_1^i+\cC_2^i)\|_{L^2(\R^2\times\T)}\ls\ve_1^3\w{t}^{\ve_2}.
\end{equation}

Next, we treat the higher order nonlinear terms $G^i=G_4^i(\p^{\le3}u)+\tilde G^i(\p^{\le4}u)$ in \eqref{QWE3}.
According to the definition \eqref{App:A3} of Appendix, one has
\begin{equation}\label{energy:V:pf9}
\begin{split}
|Z^aG_4^i(\p^{\le3}u)|&\ls\sum_{b+c+d+e\le a}|\p^{\le2}\p Z^bu||\p^{\le2}\p Z^cu|\\
&\times(|Z^du||Z^eu|+|Z^du||\p^{\le2}\p Z^eu|+|\p^{\le2}\p Z^du||\p^{\le2}\p Z^eu|).
\end{split}
\end{equation}

We firstly deal with the term $|\p^{\le2}\p Z^bu||\p^{\le2}\p Z^cu||Z^du||Z^eu|$
in the right hand side of \eqref{energy:V:pf9}.
By symmetry, one can assume $|b|\ge|c|$ and $|d|\ge|e|$ without loss of generality.
From $|b|+|c|+|d|+|e|\le|a|\le2N-3$, we have $|c|,|e|\le N-2$.
Thus, it follows from \eqref{BA:pw:full} that
\begin{equation}\label{energy:V:pf10}
|\p^{\le2}\p Z^cu||Z^eu|\ls\ve_1^2\w{t+|x|}^{-1}\w{t-|x|}^{-0.6}.
\end{equation}

When $|d|\le N+7$, applying \eqref{BA:pw:full} again yields
\begin{equation}\label{energy:V:pf11}
\begin{split}
&\sum_{\substack{|b|+|c|+|d|+|e|\le2N-3,\\|d|\le N+7}}\|\w{t+|x|}|\p^{\le2}\p Z^bu||\p^{\le2}\p Z^cu||Z^du||Z^eu|\|_{L^2(\R^2\times\T)}\\
&\ls\ve_1^3\sum_{|b|\le2N-3}\|\w{t+|x|}^{-1/2}\p^{\le2}\p Z^bu\|_{L^2(\R^2\times\T)}\\
&\ls\ve_1^3\w{t}^{-1/2}E_{2N}[u](t)\ls\ve_1^4\w{t}^{\ve_2-1/2},
\end{split}
\end{equation}
where \eqref{energy:V:pf10} has been used.

When  $|d|\ge N+8$, one has $|b|\le N-11$.
Note that for $x\in\supp(1-\chi(\frac{|x|}{\w{t}}))$, $\w{t-|x|}\approx\w{t+|x|}$ holds.
Then it can be concluded from \eqref{Hardy:ineq}, \eqref{mod:Hardy:ineq}, \eqref{BA:pw:full} and \eqref{energy:V:pf10} that
\begin{equation}\label{energy:V:pf12}
\begin{split}
&\sum_{\substack{|b|+|c|+|d|+|e|\le2N-3,\\|d|\ge N+8}}\|\w{t+|x|}|\p^{\le2}\p Z^bu||\p^{\le2}\p Z^cu||Z^du||Z^eu|\|_{L^2(\R^2\times\T)}\\
&\ls\ve_1^3\sum_{|d|\le2N-3}\|\w{t+|x|}^{-1/2}\w{t-|x}^{-0.9}Z^du\|_{L^2(\R^2\times\T)}\\
&\ls\ve_1^3\sum_{|d|\le2N-3}\Big(\w{t}^{-1/2}\Big\|\frac{\chi(\frac{|x|}{\w{t}})Z^du}{\w{t-|x|}^{0.9}}\Big\|_{L^2(\R^2\times\T)}
+\Big\|\frac{(1-\chi(\frac{|x|}{\w{t}}))Z^du}{|x|\ln|x|}\Big\|_{L^2(\R^2\times\T)}\Big)\\
&\ls\ve_1^3\Big(\ln(2+t)E_{2N}[u](t)+\sum_{|d|\le2N-3}\Big\|\frac{Z^du}{|x|\ln|x|}\Big\|_{L^2(\R^2\times\T)}\Big)\\
&\ls\ve_1^3\w{t}^{\ve_2}\ln(2+t),
\end{split}
\end{equation}
where $\chi$ is defined by \eqref{cutoff:def}.
In view of \eqref{BA2}, \eqref{BA3} and \eqref{BA5}, the estimates on the other terms in \eqref{energy:V:pf9} and $Z^a\tilde G^i(\p^{\le4}u)$
can be treated easily.
Then \eqref{energy:V:pf1} is achieved by \eqref{energy:V:pf8}, \eqref{energy:V:pf9}, \eqref{energy:V:pf11} and \eqref{energy:V:pf12}.

Secondly, we prove \eqref{energy:V} with the help of \eqref{energy:V:pf1}.
Due to
\begin{equation*}
\frac12\p_t(|\p Z^aV|^2)-\sum_{i=1}^3\p_i(\p_tZ^aV\p_iZ^aV)=\p_tZ^aV\Box Z^aV,
\end{equation*}
then by integrating this identity over $[0,t]\times\R^2\times\T$, one has that
\begin{equation}\label{energy:ineq}
\|\p Z^aV(t)\|^2_{L^2_{x,y}}\ls\|\p Z^aV(0)\|^2_{L^2_{x,y}}+\int_0^t\|\p_tZ^aV(s)\|_{L^2_{x,y}}\|Z^a\Box V(s)\|_{L^2_{x,y}}ds
\end{equation}
This, together with \eqref{initial:data}, \eqref{normal:form} and \eqref{energy:V:pf1}, yields
\begin{equation*}
\begin{split}
\big(E_{2N-3}[V](t)\big)^2
&\ls\big(E_{2N-3}[V](0)\big)^2+\ve_1^3\int_0^t(1+s)^{2\ve_2-1}ds\sup_{s\in[0,t]}E_{2N-3}[V](s)\\
&\ls(\ve^2+\ve_2^{-1}\ve_1^3(1+t)^{2\ve_2})\sup_{s\in[0,t]}E_{2N-3}[V](s).
\end{split}
\end{equation*}
Thus,  the estimate  $E_{2N-3}[V](t)$ in \eqref{energy:V} is proved.

Next, we estimate $\cX_{2N-3}[V](t)$.
Recall Lemma 5.1 of \cite{HTY24}, for smooth function $f(t,x)$, $x\in\R^2$, it holds that
\begin{equation*}
\|\w{t-|x|}\p_{t,x}^2f\|_{L^2(\R^2)}
\ls\sum_{|a|\le1}\|\p_{t,x}\Gamma^af\|_{L^2(\R^2)}+\|\w{t+|x|}\Box_{t,x}f\|_{L^2(\R^2)}
\end{equation*}
This, together with the Sobolev embedding theorem on $\T$, \eqref{proj:property} and \eqref{energy:V:pf1}, yields
\begin{equation*}
\cX_{2N-3}[V](t)\ls E_{2N-3}[V](t)+\sum_{|a|\le2N-3}\|\w{t+|x|}\Box Z^aV(t,x,y)\|_{L^2(\R^2\times\T)}\ls\ve_1\w{t}^{2\ve_2}.
\end{equation*}
Hence the proof of \eqref{energy:V} is completed.
\end{proof}

\begin{lemma}\label{lem:pw0:high}
Let $u(t,x,y)$ be the solution of \eqref{QWE} and suppose that \eqref{BA1}-\eqref{BA5} hold.
Then for any multi-index $a$ with $|a|\le2N-5$, one has
\begin{equation}\label{pw0:high}
|P_{=0}Z^au(t,x,y)|+\w{x}^{1/2}\w{t-|x|}^{1/2}|P_{=0}\p Z^au(t,x,y)|\ls\ve_1\w{t}^{2\ve_2}.
\end{equation}
\end{lemma}
\begin{proof}
Set
\begin{equation}\label{pw0:pf1}
\cM(t):=\sum_{|a|\le2N-5}\sup_{x}(|P_{=0}Z^au(t,x,y)|+\w{x}^{1/2}\w{t-|x|}^{1/2}|P_{=0}\p Z^au(t,x,y)|).
\end{equation}
According to the definition \eqref{normal:form} and \eqref{proj:property}, we have
\begin{equation}\label{pw0:pf2}
\begin{split}
|P_{=0}Z^a(u-V)|&\ls\sum_{|b|+|c|\le|a|}(|\p^{\le1}P_{=0}Z^bu||\p^{\le1}P_{=0}Z^cu|
+\|\p^{\le1}P_{\neq0}Z^bu\|_{L_y^\infty}\|\p^{\le1}P_{\neq0}Z^cu\|_{L_y^\infty})\\
&\qquad+\sum_{|b|+|c|+|d|\le|a|}\|\p^{\le2}P_{\neq0}Z^bu\|_{L_y^\infty}\|\p^{\le2}P_{\neq0}Z^cu\|_{L_y^\infty}|P_{=0}Z^du|.
\end{split}
\end{equation}
It follows from \eqref{pw0:V1} and \eqref{energy:V} with $|a|\le2N-5$ that
\begin{equation}\label{pw0:pf3}
|P_{=0}\p Z^aV(t,x,y)|\ls\ve_1\w{t}^{2\ve_2}\w{x}^{-1/2}\w{t-|x|}^{-1/2}.
\end{equation}
Note that for $|x|\ge2\w{t}$, $\w{s-|x|}\approx\w{x}\ge\w{t}$ holds as $s\in[0,t]$.
By \eqref{initial:data} and \eqref{pw0:pf3}, one obtains that for $|x|\ge2\w{t}$,
\begin{equation}\label{pw0:pf4}
\begin{split}
\quad|P_{=0}Z^aV(t,x,y)|
&\le |P_{=0}Z^aV(0,x,y)|+\int_{0}^t|P_{=0}\p Z^aV(s,x,y)|ds\\
&\ls\ve+\ve_1\w{t}^{2\ve_2}\int_{0}^t\frac{ds}{\w{x}^{1/2}\w{s-|x}^{1/2}}\\
&\ls\ve_1\w{t}^{2\ve_2}.
\end{split}
\end{equation}
For $|x|<2\w{t}$ and $|x|\neq0$, we have from \eqref{pw0:pf3} and \eqref{pw0:pf4} that
\begin{equation}\label{pw0:pf5}
\begin{split}
|P_{=0}Z^aV(t,x,y)|&\ls|P_{=0}Z^aV(t,2\w{t}\frac{x}{|x|},y)|+\int_{|x|}^{2\w{t}}|P_{=0}\p Z^aV(t,r\frac{x}{|x|},y)|dr\\
&\ls\ve_1\w{t}^{2\ve_2}+\ve_1\w{t}^{2\ve_2}\int_0^{2\w{t}}\w{r}^{-1/2}\w{t-r}^{-1/2}dr\\
&\ls\ve_1\w{t}^{2\ve_2}+\ve_1\w{t}^{2\ve_2-1/2}\Big(\int_0^{\w{t}/2}\frac{dr}{\w{r}^{1/2}}
+\int_{\w{t}/2}^{2\w{t}}\frac{dr}{\w{t-r}^{1/2}}\Big)\\
&\ls\ve_1\w{t}^{2\ve_2}.
\end{split}
\end{equation}
Similarly, for $x=(x_1,x_2)=0$,
\begin{equation*}
\begin{split}
|P_{=0}Z^aV(t,0,0,y)|&\ls|P_{=0}Z^aV(t,2\w{t},0,y)|+\int_0^{2\w{t}}|P_{=0}\p_1Z^aV(t,x_1,0,y)|dx_1\\
&\ls\ve_1\w{t}^{2\ve_2}+\ve_1\w{t}^{2\ve_2}\int_0^{2\w{t}}\w{x_1}^{-1/2}\w{t-x_1}^{-1/2}dx_1\\
&\ls\ve_1\w{t}^{2\ve_2}.
\end{split}
\end{equation*}
This, together with \eqref{pw0:pf4} and \eqref{pw0:pf5}, yields
\begin{equation}\label{pw0:pf6}
\sum_{|a|\le2N-5}|P_{=0}Z^aV|\ls\ve_1\w{t}^{2\ve_2}.
\end{equation}

Next, we turn to the estimate on the right hand side of \eqref{pw0:pf2}.
Note that for the first line in \eqref{pw0:pf2}, $|b|+|c|\le|a|\le2N-5$ ensures $|b|\le N$ or $|c|\le N$.
Thus, collecting \eqref{BA1}, \eqref{BA2}, \eqref{BA3}, \eqref{BA5} and \eqref{pw:high} leads to
\begin{equation}\label{pw0:pf7}
\begin{split}
&\quad|\p^{\le1}P_{=0}Z^bu||\p^{\le1}P_{=0}Z^cu|
+\|\p^{\le1}P_{\neq0}Z^bu\|_{L_y^\infty}\|\p^{\le1}P_{\neq0}Z^cu\|_{L_y^\infty}\\
&\ls\ve_1\w{t}^{-1/2}\cM(t)+\ve_1^{2}\w{t}^{\ve_2-1}\w{x}^{-1/2}.
\end{split}
\end{equation}
Analogously,  for the second line in \eqref{pw0:pf2}, the same estimate as \eqref{pw0:pf7} can be obtained.
Therefore, by \eqref{pw0:pf2}, \eqref{pw0:pf6} and \eqref{pw0:pf7}, we have
\begin{equation}\label{pw0:pf8}
\sum_{|a|\le2N-5}|P_{=0}Z^au|\ls\ve_1\w{t}^{2\ve_2}+\ve_1\w{t}^{-1/2}\cM(t).
\end{equation}

Similarly to \eqref{pw0:pf2}, one has
\begin{equation}\label{pw0:pf9}
\begin{split}
|P_{=0}\p Z^a(u-V)|&\ls\sum_{|b|+|c|\le|a|}(|\p^{\le1}P_{=0}\p Z^bu||\p^{\le1}P_{=0}Z^cu|
+\|\p^{\le2}P_{\neq0}Z^bu\|_{L_y^\infty}\|\p^{\le1}P_{\neq0}Z^cu\|_{L_y^\infty})\\
&\quad+\sum_{|b|+|c|+|d|\le|a|}\|\p^{\le3}P_{\neq0}Z^bu\|_{L_y^\infty}\|\p^{\le3}P_{\neq0}Z^cu\|_{L_y^\infty}|\p^{\le1}P_{=0}Z^du|.
\end{split}
\end{equation}
It is noted that for the first line in \eqref{pw0:pf9},  $|b|\le N$ or $|c|\le N$ always hold.
When $|b|\le N$, it can be deduced from \eqref{BA3} and \eqref{pw0:pf1} that
\begin{equation}\label{pw0:pf10}
|\p^{\le1}P_{=0}\p Z^bu||\p^{\le1}P_{=0}Z^cu|\ls\ve_1\w{x}^{-1/2}\w{t-|x|}^{-1.3}\cM(t).
\end{equation}
When $|c|\le N$, \eqref{BA1}, \eqref{BA2} and \eqref{pw:high} imply
\begin{equation}\label{pw0:pf11}
|\p^{\le1}P_{=0}\p Z^bu||\p^{\le1}P_{=0}Z^cu|
\ls\ve_1^2\w{t+|x|}^{-1/2}\w{t-|x|}^{-0.3}\w{t}^{\ve_2}\w{x}^{-1/2}.
\end{equation}
For the last term in the first line of \eqref{pw0:pf9}, it follows from \eqref{BA1}, \eqref{BA5} and \eqref{pw:high} that
\begin{equation}\label{pw0:pf12}
\|\p^{\le2}P_{\neq0}Z^bu\|_{L_y^\infty}\|\p^{\le1}P_{\neq0}Z^cu\|_{L_y^\infty}
\ls\ve_1^2\w{t+|x|}^{-1}\w{t}^{\ve_2}\w{x}^{-1/2}
\ls\ve_1^2\w{x}^{-1/2}\w{t+|x|}^{\ve_2-1}.
\end{equation}
The estimate on the second line in \eqref{pw0:pf9} can be treated analogously.
Thus, we have from \eqref{pw0:pf1}, \eqref{pw0:pf3} and \eqref{pw0:pf8}-\eqref{pw0:pf12} that
\begin{equation*}
\cM(t)\ls\ve_1\w{t}^{2\ve_2}+\ve_1\w{t}^{-1/2}\cM(t).
\end{equation*}
Together with the smallness of $\ve_1$, this yields \eqref{pw0:high}.
\end{proof}

\begin{corollary}
Let $u(t,x,y)$ be the solution of \eqref{QWE} and suppose that \eqref{BA1}-\eqref{BA5} hold.
Then
\begin{equation}\label{pw0:good}
\sum_{|a|\le2N-6}|P_{=0}\bar\p Z^au(t,x,y)|\ls\ve_1\w{x}^{-1/2}\w{t+|x|}^{-1/2}\w{t}^{2\ve_2}.
\end{equation}
\end{corollary}
\begin{proof}
\eqref{pw0:good} in the case of $|x|\le\w{t}/2$ can be obtained by \eqref{pw0:high}.

For $|x|\ge\w{t}/2$, it follows from \eqref{good:ident} and \eqref{pw0:high} that
\begin{equation*}
\begin{split}
\w{x}\sum_{|a|\le2N-6}|P_{=0}\bar\p Z^au(t,x,y)|
&\ls\sum_{|b|\le2N-5}|P_{=0}Z^bu|+\w{t-|x|}\sum_{|a|\le2N-6}|P_{=0}\p Z^au|\\
&\ls\ve_1\w{t}^{2\ve_2}.
\end{split}
\end{equation*}
This completes the proof of \eqref{pw0:good}.
\end{proof}

\subsection{Lower order pointwise estimates of the zero modes}
Based on the estimates in the last subsection, we are able to improve the pointwise estimates of the zero
modes in \eqref{BA2} and \eqref{BA3}.
First, recall two lemmas from \cite[Lemmas 4.1, 4.2]{Kubo19}.
\begin{lemma}\label{lem:improv1}
Let $w(t,x)$ be a solution of $\Box_{t,x}w=0$ with the initial data $(w,\p_tw)|_{t=0}=(f(x),g(x))$.
Then
\begin{equation*}
\begin{split}
\w{t+|x|}^{1/2}\w{t-|x|}^{1/2}|w(t,x)|&\ls\cA_{5/2,0}[f,g],\\
\w{t+|x|}^{1/2}\w{t-|x|}^{3/2}|\p w(t,x)|&\ls\cA_{7/2,1}[f,g],
\end{split}
\end{equation*}
where $\ds\cA_{\kappa,s}[f,g]:=\sum_{\tilde\Gamma\in\{\p_1,\p_2,\Omega\}}
(\sum_{|a|\le s+1}\|\w{z}^\kappa\tilde\Gamma^af(z)\|_{L^\infty}+\sum_{|a|\le s}\|\w{z}^\kappa\tilde\Gamma^ag(z)\|_{L^\infty})$.
\end{lemma}

\begin{lemma}\label{lem:improv2}
Let $w(t,x)$ be a solution of $\Box_{t,x}w=\cF(t,x)$ with zero initial data $(w,\p_tw)|_{t=0}=(0,0)$.
Then for $\rho\in(0,1/2)$ and $\kappa>0$, one has
\begin{equation*}
\begin{split}
\w{t+|x|}^{1/2}\w{t-|x|}^{\rho}|w(t,x)|&\ls\sup_{s,z}(\w{z}^{1/2}\cW_{1+\rho,1+\kappa}(s,z)|\cF(s,z)|),\\
\w{x}^{1/2}\w{t-|x|}^{1+\rho}|\p w(t,x)|&\ls\sum_{|a|\le1}\sup_{s,z}(\w{z}^{1/2}\cW_{1+\rho+\kappa,1}(s,z)|\Gamma^a\cF(s,z)|),
\end{split}
\end{equation*}
where $\cW_{\sigma,\lambda}(t,x):=\w{t+|x|}^\sigma(\min\{\w{x},\w{t-|x|}\})^\lambda$.
\end{lemma}

\begin{lemma}\label{lem:pw0:improv}
Let $u$ be the solution of \eqref{QWE} and suppose that \eqref{BA1}-\eqref{BA5} hold.
Then we have
\begin{equation}\label{pw0:improv}
\begin{split}
\w{t+|x|}^{1/2}\w{t-|x|}^{0.3}\sum_{|a|\le N+7}|P_{=0}Z^au(t,x,y)|&\ls\ve+\ve_1^2,\\
\w{x}^{1/2}\w{t-|x|}^{1.3}\sum_{|a|\le N+6}|P_{=0}\p Z^au(t,x,y)|&\ls\ve+\ve_1^2.
\end{split}
\end{equation}
\end{lemma}
\begin{proof}
To prove \eqref{pw0:improv}, we now establish the following estimate
\begin{equation}\label{pw0:improv1}
\sum_{|a|\le N-1}|P_{=0}Z^aV(t,x,y)|\ls\ve_1\w{t+|x|}^{-1/2}\w{t-|x|}^{-0.4}.
\end{equation}
Applying Lemmas \ref{lem:improv1}, \ref{lem:improv2} with $\kappa=0.01$, $\rho=0.4$ and $\rho=0.3$, respectively, to $\Box Z^aV$
derives
\begin{equation}\label{pw0:improv2}
\begin{split}
&\w{t+|x|}^{1/2}\w{t-|x|}^{0.4}\sum_{|a|\le N-1}|P_{=0}Z^aV|\\
&\ls\ve+\sup_{s,z}\sum_{|a|\le N-1}\w{z}^{1/2}\cW_{1.41,1}(s,z)|P_{=0}\Box Z^aV(s,z)|
\end{split}
\end{equation}
and
\begin{equation}\label{pw0:improv3}
\begin{split}
&\w{t+|x|}^{1/2}\w{t-|x|}^{0.3}\sum_{|a|\le N+7}|P_{=0}Z^aV|+\w{x}^{1/2}\w{t-|x|}^{1.3}\sum_{|a|\le N+6}|P_{=0}\p Z^aV|\\
&\ls\ve+\sup_{s,z}\sum_{|a|\le N+7}\w{z}^{1/2}\cW_{1.31,1}(s,z)|P_{=0}\Box Z^aV(s,z)|,
\end{split}
\end{equation}
where \eqref{initial:data} has been used.

Next, we deal with the right hand sides of \eqref{pw0:improv2} and \eqref{pw0:improv3}.
From \eqref{eqn:normal}, one has $\Box Z^aV^i=Z^a(\cC_1^i+\cC_2^i+G_4^i(\p^{\le3}u)+\tilde G^i(\p^{\le4}u))$.

We only deal with the term $\ds\sum_{j,k,l=1}^m\sum_{\alpha,\beta,\mu,\nu=0}^3\tilde Q_{ijkl}^{\alpha\beta\mu\nu}\p^2_{\alpha\beta}Z^bu^j\p_{\mu}Z^cu^k\p_{\nu}Z^du^l$ in $Z^a\cC_1^i$
since  the other term can be treated analogously.
It follows from \eqref{proj:property} that
\begin{equation}\label{pw0:improv4}
\begin{split}
&\p^2_{\alpha\beta}Z^bu^j\p_{\mu}Z^cu^k\p_{\nu}Z^du^l
=P_{=0}\p^2_{\alpha\beta}Z^bu^jP_{=0}\p_{\mu}Z^cu^kP_{=0}\p_{\nu}Z^du^l\\
&+P_{\neq0}\p^2_{\alpha\beta}Z^bu^jP_{=0}\p_{\mu}Z^cu^kP_{=0}\p_{\nu}Z^du^l
+P_{=0}\p^2_{\alpha\beta}Z^bu^jP_{\neq0}\p_{\mu}Z^cu^kP_{=0}\p_{\nu}Z^du^l\\
&+P_{=0}\p^2_{\alpha\beta}Z^bu^jP_{=0}\p_{\mu}Z^cu^kP_{\neq0}\p_{\nu}Z^du^l\\
&+P_{\neq0}\p^2_{\alpha\beta}Z^bu^jP_{\neq0}\p_{\mu}Z^cu^kP_{=0}\p_{\nu}Z^du^l
+P_{\neq0}\p^2_{\alpha\beta}Z^bu^jP_{=0}\p_{\mu}Z^cu^kP_{\neq0}\p_{\nu}Z^du^l\\
&+P_{=0}\p^2_{\alpha\beta}Z^bu^jP_{\neq0}\p_{\mu}Z^cu^kP_{\neq0}\p_{\nu}Z^du^l\\
&+P_{\neq0}\p^2_{\alpha\beta}Z^bu^jP_{\neq0}\p_{\mu}Z^cu^kP_{\neq0}\p_{\nu}Z^du^l.
\end{split}
\end{equation}
With respect to the last term in the first line of \eqref{pw0:improv4}, by \eqref{proj:property}, one has
\begin{equation*}
\begin{split}
\sum_{\alpha,\beta,\mu,\nu=0}^3\tilde Q_{ijkl,abcd}^{\alpha\beta\mu\nu}P_{=0}\p^2_{\alpha\beta}Z^bu^jP_{=0}\p_{\mu}Z^cu^kP_{=0}\p_{\nu}Z^du^l\\
=\sum_{\alpha,\beta,\mu,\nu=0}^2\tilde Q_{ijkl,abcd}^{\alpha\beta\mu\nu}P_{=0}\p^2_{\alpha\beta}Z^bu^jP_{=0}\p_{\mu}Z^cu^kP_{=0}\p_{\nu}Z^du^l.
\end{split}
\end{equation*}
This, together with \eqref{null:structure}, yields
\begin{equation}\label{pw0:improv5}
\begin{split}
&\Big|\sum_{\alpha,\beta,\mu,\nu=0}^3\tilde Q_{ijkl,abcd}^{\alpha\beta\mu\nu}P_{=0}\p^2_{\alpha\beta}Z^bu^jP_{=0}\p_{\mu}Z^cu^kP_{=0}\p_{\nu}Z^du^l\Big|\\
&\ls|P_{=0}\bar\p\p Z^bu^j||P_{=0}\p Z^cu^k||P_{=0}\p Z^du^l|+|P_{=0}\p^2Z^bu^j||P_{=0}\bar\p Z^cu^k||P_{=0}\p Z^du^l|\\
&\quad+|P_{=0}\p^2Z^bu^j|P_{=0}\p Z^cu^k||P_{=0}\bar\p Z^du^l|.
\end{split}
\end{equation}
Due to $|b|+|c|+|d|\le|a|\le N+7\le2N-7$, then there is at least two factors in the terms of the right hand side
of \eqref{pw0:improv5} that
fulfill the estimates \eqref{BA3} and \eqref{pw0:good:low}.
While, the remaining one factor satisfies the estimate \eqref{pw0:high} or \eqref{pw0:good}.
Hence,
\begin{equation}\label{pw0:improv6}
\begin{split}
&\sum_{|b|+|c|+|d|\le N+7}\Big|\sum_{\alpha,\beta,\mu,\nu=0}^3\tilde Q_{ijkl,abcd}^{\alpha\beta\mu\nu}
P_{=0}\p^2_{\alpha\beta}Z^bu^jP_{=0}\p_{\mu}Z^cu^kP_{=0}\p_{\nu}Z^du^l\Big|\\
&\ls\ve_1^3\w{t+|x|}^{2\ve_2-1/2}\w{x}^{-3/2}\w{t-|x|}^{-2.6}+\ve_1^3\w{t+|x|}^{2\ve_2-1}\w{x}^{-3/2}\w{t-|x|}^{-2.1}\\
&\ls\ve_1^3\w{t+|x|}^{2\ve_2-1}\w{x}^{-1}\w{t-|x|}^{-2.1}.
\end{split}
\end{equation}
For the terms from the second to fifth lines in \eqref{pw0:improv4},  by $N+7\le2N-10$,
it is easy to know that they contain at least two factors that fulfill \eqref{BA3} and \eqref{BA5},
and one remaining factor that satisfies \eqref{BA4} or \eqref{pw0:high}.
Then one has
\begin{equation}\label{pw0:improv7}
\begin{split}
&\sum_{|b|+|c|+|d|\le N+7}(|P_{\neq0}\p^2_{\alpha\beta}Z^bu^jP_{=0}\p_{\mu}Z^cu^kP_{=0}\p_{\nu}Z^du^l|\\
&\quad+|P_{=0}\p^2_{\alpha\beta}Z^bu^jP_{\neq0}\p_{\mu}Z^cu^kP_{=0}\p_{\nu}Z^du^l|
+|P_{=0}\p^2_{\alpha\beta}Z^bu^jP_{=0}\p_{\mu}Z^cu^kP_{\neq0}\p_{\nu}Z^du^l|)\\
&\ls\ve_1^3\w{x}^{-1}\w{t-|x|}^{-2.6}\w{t+|x|}^{\ve_2-1/2}(\w{t+|x|}^{-0.45}+\w{x}^{-1/2}\w{t-|x|}^{-0.4})\\
&\quad+\ve_1^3\w{t+|x|}^{2\ve_2-1}\w{x}^{-1}\w{t-|x|}^{-1.8}\\
&\ls\ve_1^3\w{x}^{-1}\w{t-|x|}^{-1.8}\w{t+|x|}^{\ve_2-0.95}
\end{split}
\end{equation}
and
\begin{equation}\label{pw0:improv8}
\begin{split}
&\sum_{|b|+|c|+|d|\le N+7}(|P_{\neq0}\p^2_{\alpha\beta}Z^bu^jP_{\neq0}\p_{\mu}Z^cu^kP_{=0}\p_{\nu}Z^du^l|\\
&\quad+|P_{\neq0}\p^2_{\alpha\beta}Z^bu^jP_{=0}\p_{\mu}Z^cu^kP_{\neq0}\p_{\nu}Z^du^l|
+|P_{=0}\p^2_{\alpha\beta}Z^bu^jP_{\neq0}\p_{\mu}Z^cu^kP_{\neq0}\p_{\nu}Z^du^l|)\\
&\ls\ve_1^3\w{x}^{-1/2}\w{t-|x|}^{-1.3}\w{t+|x|}^{\ve_2-3/2}(\w{t+|x|}^{-0.45}+\w{x}^{-1/2}\w{t-|x|}^{-0.4})\\
&\quad+\ve_1^3\w{x}^{-1/2}\w{t-|x|}^{-1/2}\w{t+|x|}^{2\ve_2-2}\\
&\ls\ve_1^3\w{x}^{-1/2}\w{t-|x|}^{-1/2}\w{t+|x|}^{\ve_2-1.95}.
\end{split}
\end{equation}
For the term in the last line of \eqref{pw0:improv4}, we can analogously obtain
\begin{equation}\label{pw0:improv9}
\begin{split}
&\sum_{|b|+|c|+|d|\le N+7}|P_{\neq0}\p^2_{\alpha\beta}Z^bu^jP_{\neq0}\p_{\mu}Z^cu^kP_{\neq0}\p_{\nu}Z^du^l|\\
&\ls\ve_1^3\w{t+|x|}^{\ve_2-5/2}(\w{t+|x|}^{-0.45}+\w{x}^{-1/2}\w{t-|x|}^{-0.4}).
\end{split}
\end{equation}
Collecting \eqref{pw0:improv4}-\eqref{pw0:improv9} with \eqref{proj:property} leads to
\begin{equation}\label{pw0:improv10}
\sum_{|a|\le N+7}\w{x}^{1/2}\w{t+|x|}^{1.41}\min\{\w{x},\w{t-|x|}\}|P_{=0}Z^a\cC_1^i|\ls\ve_1^3.
\end{equation}
For $Z^a\cC_2^i$, it follows from \eqref{energy:V:pf3} that
\begin{equation}\label{pw0:improv11}
\begin{split}
&|Z^a\cC_2^i|\ls\sum_{|b|+|c|+|d|\le|a|}\Big\{|P_{=0}Z^bu||\p^{\le2}P_{=0}\p Z^cu||P_{=0}\bar\p \p^{\le2}Z^du|\\
&+|P_{\neq0}Z^bu||\p^{\le2}P_{=0}\p Z^cu||\p^{\le2}P_{=0}\p Z^du|
+|P_{\neq0}Z^bu||\p^{\le2}P_{\neq0}\p Z^cu||\p^{\le2}P_{\neq0}\p Z^du|\\
&+|P_{\neq0}Z^bu||\p^{\le2}P_{\neq0}\p Z^cu||\p^{\le2}P_{=0}\p Z^du|
+|P_{=0}\p Z^bu||\p^{\le3}P_{\neq0}Z^cu||\p^{\le3}P_{\neq0}Z^du|\\
&+|P_{=0}Z^bu||\p^{\le2}P_{=0}\p Z^cu||\p^{\le3}P_{\neq0}Z^du|\Big\},
\end{split}
\end{equation}
where  \eqref{proj:property} has been used.

With respect to the term in the first line of \eqref{pw0:improv11}, when $|b|+|c|+|d|\le|a|\le N+7$,
then $|c|\le N-4$ or $|d|\le N-4$ holds and further \eqref{BA3} or \eqref{pw0:good:low} can be applied.
For the other left cases of $b, c$ and $d$, the estimate \eqref{pw0:high} or \eqref{pw0:good} hold.
Together with \eqref{BA2}, this derives
\begin{equation}\label{pw0:improv12}
\begin{split}
&\sum_{|b|+|c|+|d|\le N+7}|P_{=0}Z^bu||\p^{\le2}P_{=0}\p Z^cu||P_{=0}\bar\p \p^{\le2}Z^du|\\
&\ls\ve_1^3\w{t+|x|}^{2\ve_2-1}\w{x}^{-1}\w{t-|x|}^{-1.6}+\ve_1^3\w{t+|x|}^{2\ve_2-3/2}\w{x}^{-1}\w{t-|x|}^{-1.1}\\
&\ls\ve_1^3\w{t+|x|}^{2\ve_2-1}\w{x}^{-1}\w{t-|x|}^{-1.6}.
\end{split}
\end{equation}
The treatment on the terms in the second and third lines of \eqref{pw0:improv11} are the same
as in \eqref{pw0:improv7}-\eqref{pw0:improv9}, we then have
\begin{equation}\label{pw0:improv13}
\begin{split}
&\sum_{|b|+|c|+|d|\le N+7}\Big\{
|P_{\neq0}Z^bu||\p^{\le2}P_{=0}\p Z^cu||\p^{\le2}P_{=0}\p Z^du|
+|P_{\neq0}Z^bu||\p^{\le2}P_{\neq0}\p Z^cu||\p^{\le2}P_{\neq0}\p Z^du|\\
&+|P_{\neq0}Z^bu||\p^{\le2}P_{\neq0}\p Z^cu||\p^{\le2}P_{=0}\p Z^du|
+|P_{=0}\p Z^bu||\p^{\le3}P_{\neq0}Z^cu||\p^{\le3}P_{\neq0}Z^du|\Big\}\\
&\ls\ve_1^3\w{x}^{-1}\w{t-|x|}^{-1.8}\w{t+|x|}^{\ve_2-0.95}+\ve_1^3\w{x}^{-1/2}\w{t-|x|}^{-1/2}\w{t+|x|}^{\ve_2-1.95}\\
&\quad+\ve_1^3\w{t+|x|}^{\ve_2-5/2}(\w{t+|x|}^{-0.45}+\w{x}^{-1/2}\w{t-|x|}^{-0.4}).
\end{split}
\end{equation}
For the term in the last line of \eqref{pw0:improv11}, it can be deduced from \eqref{BA2}, \eqref{BA3} and \eqref{BA5} that
\begin{equation}\label{pw0:improv14}
\begin{split}
&\sum_{|b|+|c|+|d|\le N-1}|P_{=0}Z^bu||\p^{\le2}P_{=0}\p Z^cu||\p^{\le3}P_{\neq0}Z^du|\\
&\ls\ve_1^3\w{x}^{-1/2}\w{t-|x|}^{-1.6}\w{t+|x|}^{-3/2}.
\end{split}
\end{equation}
Collecting \eqref{pw0:improv11}-\eqref{pw0:improv14} yields
\begin{equation}\label{pw0:improv15}
\sum_{|a|\le N-1}\w{x}^{1/2}\w{t+|x|}^{1.41}\min\{\w{x},\w{t-|x|}\}|P_{=0}Z^a\cC_2^i|\ls\ve_1^3.
\end{equation}

Next, we handle $Z^aG_4^i(\p^{\le3}u)$ with $|a|\le N+7$.
It follows from \eqref{proj:property}, \eqref{BA4} and \eqref{pw0:high} that
\begin{equation}\label{pw0:improv16}
\sum_{|a|\le2N-9}|\p Z^au|\ls\ve_1\w{t+|x|}^{2\ve_2}\w{x}^{-1/2}\w{t-|x|}^{-0.3}.
\end{equation}
Applying \eqref{BA:pw:full} and \eqref{pw0:improv16} to \eqref{energy:V:pf9} derives
\begin{equation}\label{pw0:improv17}
\sum_{|a|\le N+7}|Z^aG_4^i(\p^{\le3}u)|\ls\ve_1^4\w{t+|x|}^{2\ve_2-3/2}\w{x}^{-1/2}\w{t-|x|}^{-1.2}.
\end{equation}
Analogously, $Z^a\tilde G^i(\p^{\le4}u)$ can be estimated as for \eqref{pw0:improv17}.

Therefore, \eqref{pw0:improv1} is achieved by \eqref{pw0:improv2}, \eqref{pw0:improv10}, \eqref{pw0:improv15} and \eqref{pw0:improv17}.

We now show that \eqref{pw0:improv1} still holds when $V$ is replaced by $u$.
By \eqref{BA2}, \eqref{BA3}, \eqref{BA5} and \eqref{pw0:pf2} with $|a|\le N-1$, we have
\begin{equation}\label{pw0:low1}
\begin{split}
|P_{=0}Z^a(u-V)|&\ls\ve_1^2\w{t+|x|}^{-1}\w{t-|x|}^{-0.6}+\ve_1^2\w{x}^{-1}\w{t-|x|}^{-2.6}+\ve_1^2\w{t+|x|}^{-2}\\
&\quad+\ve_1^3\w{t+|x|}^{-5/2}\w{t-|x|}^{-0.3}.
\end{split}
\end{equation}
This, together with \eqref{pw0:improv1}, leads to
\begin{equation}\label{pw0:low2}
\sum_{|a|\le N-1}|P_{=0}Z^au|\ls\ve_1\w{t+|x|}^{-1/2}\w{t-|x|}^{-0.4}.
\end{equation}
At last, we prove \eqref{pw0:improv}.
In view of \eqref{pw0:improv3}, \eqref{pw0:improv10}-\eqref{pw0:improv13} and \eqref{pw0:improv17}, it suffices to establish the estimate \eqref{pw0:improv14} for $|b|+|c|+|d|\le N+7$.
If $|b|\ge N$, from $|c|+|d|\le7$ and \eqref{BA2}, \eqref{BA3}, \eqref{BA5}, one then has
\begin{equation}\label{pw0:low3}
\begin{split}
&\sum_{\substack{|b|+|c|+|d|\le N+7,\\|b|\ge N}}|P_{=0}Z^bu||\p^{\le2}P_{=0}\p Z^cu||\p^{\le3}P_{\neq0}Z^du|\\
&\ls\ve_1^3\w{x}^{-1/2}\w{t-|x|}^{-1.6}\w{t+|x|}^{-3/2}.
\end{split}
\end{equation}
If $|b|\le N-1$, then it can be concluded from \eqref{BA3}, \eqref{BA4}, \eqref{BA5}, \eqref{pw0:high} and \eqref{pw0:low2} that
\begin{equation*}
\begin{split}
&\sum_{\substack{|b|+|c|+|d|\le N+7,\\|b|\le N-1}}|P_{=0}Z^bu||\p^{\le2}P_{=0}\p Z^cu||\p^{\le3}P_{\neq0}Z^du|\\
&\ls\ve_1^3\w{x}^{-1/2}\w{t-|x|}^{-1.7}\w{t+|x|}^{\ve_2-1}(\w{t+|x|}^{-0.45}+\w{x}^{-1/2}\w{t-|x|}^{-0.4})\\
&\quad+\ve_1^3\w{x}^{-1/2}\w{t-|x|}^{-0.9}\w{t+|x|}^{2\ve_2-3/2}\\
&\ls\ve_1^3\w{x}^{-1/2}\w{t-|x|}^{-0.9}\w{t+|x|}^{\ve_2-1.45}.
\end{split}
\end{equation*}
This, together with \eqref{pw0:improv11}-\eqref{pw0:improv13} and \eqref{pw0:low3}, yields
\begin{equation}\label{pw0:low4}
\sum_{|a|\le N+7}\w{x}^{1/2}\w{t+|x|}^{1.31}\min\{\w{x},\w{t-|x|}\}|P_{=0}Z^a\cC_2^i|\ls\ve_1^3.
\end{equation}
Thus, collecting \eqref{pw0:improv3}, \eqref{pw0:improv10}, \eqref{pw0:improv17} and \eqref{pw0:low4} shows
\begin{equation}\label{pw0:low5}
\begin{split}
\w{t+|x|}^{1/2}\w{t-|x|}^{0.3}\sum_{|a|\le N+7}|P_{=0}Z^aV|&\ls\ve+\ve_1^2,\\
\w{x}^{1/2}\w{t-|x|}^{1.3}\sum_{|a|\le N+6}|P_{=0}\p Z^aV|&\ls\ve+\ve_1^2.
\end{split}
\end{equation}
In addition, it follows from \eqref{BA2}-\eqref{BA5}, \eqref{pw0:high}, \eqref{pw0:pf2} and \eqref{pw0:low2} that
\begin{equation}\label{pw0:low6}
\begin{split}
&\quad\sum_{|a|\le N+7}|P_{=0}Z^a(u-V)|\\
&\ls\ve_1^2\w{t+|x|}^{-1}\w{t-|x|}^{-0.6}+\ve_1^2\w{t+|x|}^{2\ve_2-1/2}\w{x}^{-1/2}\w{t-|x|}^{-0.9}\\
&+\ve_1^2\w{t}^{2\ve_2}\w{x}^{-1}\w{t-|x|}^{-1.8}+\ve_1^2\w{t+|x|}^{\ve_2-3/2}(\w{t+|x|}^{-0.45}+\w{x}^{-1/2}\w{t-|x|}^{-0.4})\\
&+\ve_1^3\w{t+|x|}^{\ve_2-2}\w{t-|x|}^{-0.3}(\w{t+|x|}^{-0.45}+\w{x}^{-1/2}\w{t-|x|}^{-0.4})\\
&\ls\ve_1^2\w{t+|x|}^{2\ve_2-1}\w{t-|x|}^{-0.3}.
\end{split}
\end{equation}
Analogously, for \eqref{pw0:pf9} with $|a|\le N+6$, by \eqref{BA2}-\eqref{BA5}, \eqref{pw0:high} and \eqref{pw0:low2}, we have
\begin{equation}\label{pw0:low7}
\begin{split}
&\quad\sum_{|a|\le N+6}|P_{=0}\p Z^a(u-V)|\\
&\ls\ve_1^2\w{t+|x|}^{2\ve_2-1/2}\w{x}^{-1/2}\w{t-|x|}^{-0.9}
+\ve_1^2\w{t+|x|}^{-1}\w{x}^{-1/2}\w{t-|x|}^{-1.6}\\
&+\ve_1^2\w{t}^{2\ve_2}\w{x}^{-1}\w{t-|x|}^{-1.8}
+\ve_1^3\w{t+|x|}^{\ve_2-3/2}(\w{t+|x|}^{-0.45}+\w{x}^{-1/2}\w{t-|x|}^{-0.4})\\
&\ls\ve_1^2\w{x}^{-1/2}\w{t-|x|}^{2\ve_2-1.4}.
\end{split}
\end{equation}
Based on \eqref{pw0:low5}, \eqref{pw0:low6} and \eqref{pw0:low7}, we have finished the proof of Lemma \ref{lem:pw0:improv}.
\end{proof}

\section{Pointwise estimates of the non-zero modes}\label{sect4}

This section aims to improve the bootstrap assumptions \eqref{BA4} and \eqref{BA5}.

\subsection{Pointwise estimates for the solution of linear wave equation on $\R^2\times\T$}

\begin{lemma}\label{lem:pwKG}
Let $w(t,x,y)$ be the solution of $~\Box w=\cF(t,x,y)$, it holds that
\begin{equation}\label{pwKG1}
\begin{split}
\w{t+|x|}^{0.95}|P_{\neq0}w(t,x,y)|&\ls\sup_{\tau\in[0,t]}\sum_{|a|\le5}\w{\tau}^{-1/20}
\|\w{\tau+|x'|}P_{\neq0}Z^a\cF(\tau,x',y)\|_{L^2(\R^2\times\T)}\\
&\qquad+\sum_{|a|\le7}\|\w{x'}^2Z^aw(0,x',y)\|_{L^2(\R^2\times\T)}
\end{split}
\end{equation}
and
\begin{equation}\label{pwKG2}
\begin{split}
\w{t+|x|}|P_{\neq0}w(t,x,y)|&\ls\sup_{\tau\in[0,t]}\sum_{|a|\le5}\w{\tau}^{1/10}
\|\w{\tau+|x'|}P_{\neq0}Z^a\cF(\tau,x',y)\|_{L^2(\R^2\times\T)}\\
&\qquad+\sum_{|a|\le7}\|\w{x'}^2Z^aw(0,x',y)\|_{L^2(\R^2\times\T)}.
\end{split}
\end{equation}
\end{lemma}
\begin{proof}
At first, we recall two estimates in \cite{HTY24}.
Let $v(t,x)$ be the solution of $\Box_{t,x}v+v=\tilde\cF(t,x)$.
By a minor modification on the proof of  Lemma 4.4 in \cite{HTY24}, we can obtain
\begin{equation}\label{pwKG:pf1}
\begin{split}
\w{t+|x|}^{0.95}|v(t,x)|&\ls\sup_{\tau\in[0,t]}\sum_{|a|\le4}\w{\tau}^{-1/20}
\|\w{\tau+|x'|}\Gamma^a\tilde\cF(\tau,x')\|_{L^2(\R^2)}\\
&\quad+\sum_{|a|\le5}\|\w{x'}^2\Gamma^av(0,x')\|_{L^2(\R^2)}
\end{split}
\end{equation}
and
\begin{equation}\label{pwKG:pf2}
\begin{split}
\w{t+|x|}|v(t,x)|&\ls\sup_{\tau\in[0,t]}\sum_{|a|\le4}\w{\tau}^{1/10}
\|\w{\tau+|x'|}\Gamma^a\tilde\cF(\tau,x')\|_{L^2(\R^2)}\\
&\quad+\sum_{|a|\le5}\|\w{x'}^2\Gamma^av(0,x')\|_{L^2(\R^2)}.
\end{split}
\end{equation}
In addition, taking the Fourier transformation on the both sides of $\Box w=\cF$ with respect to $y$ yields
\begin{equation}\label{pwKG:pf3}
(\Box_{t,x}+|n|^2)w_n(t,x)=\cF_n(t,x).
\end{equation}
Let $(s,z)=(|n|t,|n|x)$, $\Box_{s,z}=\p_s^2-\Delta_z$ and $\tilde{w}_n(s,z):=(Z^au)_n(\frac{s}{|n|},\frac{z}{|n|})$.
Then \eqref{pwKG:pf3} is changed to
\begin{equation*}
|n|^2(\Box_{s,z}+1)\tilde{w}_n(s,z)=\tilde{\cF}_n(s,z):=\cF_n(\frac{s}{|n|},\frac{z}{|n|}).
\end{equation*}
Therefore, it is derived from \eqref{pwKG:pf1} and \eqref{pwKG:pf2} that
\begin{equation*}
\begin{split}
\w{s+|z|}^{0.95}|\tilde{w}_n(s,z)|
&\ls|n|^{-2}\sup_{\tau\in[0,s]}\sum_{|a|\le4}\sum_{{\tilde\Gamma}\in\{\p_{\tau,z'},L,\Omega\}}
\w{\tau}^{-1/20}\|\w{\tau+|z'|}{\tilde\Gamma}^a\tilde{\cF}_n(\tau,z')\|_{L_{z'}^2(\R^2)}\\
&\quad+|n|^{-2}\sum_{|a|\le5}\sum_{\tilde\Gamma\in\{\p_{\tau,z'},L,\Omega\}}
\|\w{z'}^2\tilde\Gamma^a\tilde{w}_n(0,z')\|_{L_{z'}^2(\R^2)},\\
\w{s+|z|}|\tilde{w}_n(s,z)|
&\ls|n|^{-2}\sup_{\tau\in[0,s]}\sum_{|a|\le4}\sum_{\tilde\Gamma\in\{\p_{\tau,z'},L,\Omega\}}
\w{\tau}^{1/10}\|\w{\tau+|z'|}{\tilde\Gamma}^a\tilde{\cF}_n(\tau,z')\|_{L_{z'}^2(\R^2)}\\
&\quad+|n|^{-2}\sum_{|a|\le5}\sum_{\tilde\Gamma\in\{\p_{\tau,z'},L,\Omega\}}
\|\w{z'}^2\tilde\Gamma^a\tilde{w}_n(0,z')\|_{L_{z'}^2(\R^2)},
\end{split}
\end{equation*}
where $\tilde\Gamma'$s are the vector fields corresponding to $(\tau,z')$ variables.
Note that $\w{t+|x|}\ls\w{n(t+|x|)}=\w{s+|z|}$, and the vectors $L, \Omega$ are homogeneous
in $(s,z)$ and scaling invariant.
Then returning to the $(t,x)$ variables, we find that
\begin{equation}\label{pwKG:pf4}
\begin{split}
\w{t+|x|}^{0.95}|w_n(t,x)|&\ls\sup_{\tau\in[0,t]}\sum_{|a|\le4}
\w{\tau}^{-1/20}\|\w{\tau+|x|}\Gamma^a\cF_n(\tau, x)\|_{L^2_x(\R^2)}\\
&\quad+|n|\sum_{|a|\le 5}\|\w{x}^2\Gamma^aw_n(0,x)\|_{L^2_x(\R^2)},\\
\w{t+|x|}|w_n(t,x)|&\ls|n|^{1/10}\sup_{\tau\in[0,t]}\sum_{|a|\le 4}
\w{\tau}^{1/10}\|\w{\tau+|x|}\Gamma^a\cF_n(\tau, x)\|_{L^2_x(\R^2)}\\
&\quad+|n|\sum_{|a|\le5}\|\w{x}^2\Gamma^aw_n(0,x)\|_{L^2_x(\R^2)}.
\end{split}
\end{equation}
On the other hand, it is easy to find that for $C^2$ function $f(t,x,y)$ and $n\in\Z_*$,
\begin{equation*}
|f_n(t,x)|\ls|n|^{-1}\|P_{\neq0}\p_yf(t,x,y)\|_{L^2(\T)},\quad|f_n(t,x)|\ls|n|^{-2}\|P_{\neq0}\p_y^2f(t,x,y)\|_{L^2(\T)}.
\end{equation*}
This, together with \eqref{pwKG:pf4}, yields
\begin{equation*}
\begin{split}
&\w{t+|x|}^{0.95}|P_{\neq0}w(t,x,y)|\ls\w{t+|x|}^{0.95}\sum_{n\in\Z_*}|w_n(t,x)|\\
&\ls\sup_{\tau\in[0,t]}\sum_{|a|\le5}\w{\tau}^{-1/20}\|\w{\tau+|x'|}P_{\neq0}Z^a\cF(\tau,x',y)\|_{L^2(\R^2\times\T)}\\
&\quad+\sum_{|a|\le7}\|\w{x'}^2Z^aw(0,x',y)\|_{L^2(\R^2\times\T)}
\end{split}
\end{equation*}
and
\begin{equation*}
\begin{split}
\w{t+|x|}|P_{\neq0}w(t,x,y)|
&\ls\sup_{\tau\in[0,t]}\sum_{|a|\le5}\w{\tau}^{1/10}\|\w{\tau+|x'|}P_{\neq0}Z^a\cF(\tau,x',y)\|_{L^2(\R^2\times\T)}\\
&\quad+\sum_{|a|\le7}\|\w{x'}^2Z^aw(0,x',y)\|_{L^2(\R^2\times\T)}.
\end{split}
\end{equation*}
Therefore, \eqref{pwKG1} and \eqref{pwKG2} are proved.
\end{proof}

\subsection{Pointwise estimates of the non-zero modes}
In order to apply Lemma \ref{lem:pwKG} for $Z^aV$, we need establish such an estimate of $\Box Z^aV$ as:
\begin{lemma}\label{lem:NL:L2}
Let $V$ be defined by \eqref{normal:form} and suppose that \eqref{BA1}-\eqref{BA5} hold.
Then it holds that
\begin{equation}\label{NL:L2}
\sum_{|a|\le N+7}\|\w{t+|x|}P_{\neq0}\Box Z^aV(t,x,y)\|_{L^2(\R^2\times\T)}\ls\ve_1^3\w{t}^{-0.2}.
\end{equation}
\end{lemma}
\begin{proof}
According to \eqref{eqn:normal}, one has $\Box Z^aV^i=Z^a\cC_1^i+Z^a\cC_2^i+Z^aG^i$.

\vskip 0.1 true cm
{\bf $\bullet$ Estimate of $Z^a\cC_1^i$}
\vskip 0.1 true cm

Applying \eqref{BA3}, \eqref{pw0:good:low}, \eqref{pw0:good} to \eqref{pw0:improv5} yields
\begin{equation}\label{NL:L2:pf1}
\begin{split}
&\Big\|\w{t+|x|}\sum_{\alpha,\beta,\mu,\nu=0}^3Q_{ijkl,abcd}^{\alpha\beta\mu\nu}
P_{=0}\p^2_{\alpha\beta}Z^bu^jP_{=0}\p_{\mu}Z^cu^kP_{=0}\p_{\nu}Z^du^l\Big\|_{L^2_{x,y}}\\
&\ls\ve_1^2\w{t}^{2\ve_2-1/2}E_{2N}[u](t)\ls\ve_1^3\w{t}^{3\ve_2-1/2}.
\end{split}
\end{equation}
For the first term in the second line of \eqref{pw0:improv4}, from \eqref{BA3} and \eqref{BA4}, one has
\begin{equation}\label{NL:L2:pf2}
\|\w{t+|x|}P_{\neq0}\p^2_{\alpha\beta}Z^bu^jP_{=0}\p_{\mu}Z^cu^kP_{=0}\p_{\nu}Z^du^l\|_{L^2_{x,y}}
\ls\ve_1^2\w{t}^{\ve_2-0.4}E_{2N}[u](t)\ls\ve_1^3\w{t}^{2\ve_2-0.4}.
\end{equation}
For the other terms in \eqref{pw0:improv4}, we can get the same estimate as for \eqref{NL:L2:pf2}.
Then collecting \eqref{NL:L2:pf1} and \eqref{NL:L2:pf2} leads to
\begin{equation}\label{NL:L2:pf3}
\sum_{|a|\le N+7}\|\w{t+|x|}Z^a\cC_1^i\|_{L^2_{x,y}}\ls\ve_1^3\w{t}^{2\ve_2-0.4}.
\end{equation}

\vskip 0.1 true cm
{\bf $\bullet$ Estimate of $Z^a\cC_2^i$}
\vskip 0.1 true cm

From \eqref{pw0:improv11} and $|a|\le N+7$, we know $|b|+|c|+|d|\le N+7$.
For the first term in \eqref{pw0:improv11}, \eqref{BA2} and \eqref{pw0:good} yield
\begin{equation}\label{NL:L2:pf4}
\|\w{t+|x|}|P_{=0}Z^bu||\p^{\le2}P_{=0}\p Z^cu||P_{=0}\bar\p \p^{\le2}Z^du|\|_{L^2_{x,y}}\\
\ls\ve_1^3\w{t}^{3\ve_2-0.3}.
\end{equation}
The treatment on the terms in the second and third lines of \eqref{pw0:improv11} is  similar to that for \eqref{NL:L2:pf2}, then
\begin{equation}\label{NL:L2:pf5}
\|\w{t+|x|}|P_{\neq0}Z^bu||\p^{\le2}P_{=0}\p Z^cu||\p^{\le2}P_{=0}\p Z^du|\|_{L^2_{x,y}}
\ls\ve_1^3\w{t}^{2\ve_2-0.4}.
\end{equation}
For the last term in \eqref{pw0:improv11}, by \eqref{BA2} and \eqref{BA4}, we arrive at
\begin{equation*}
\|\w{t+|x|}|P_{=0}Z^bu||\p^{\le2}P_{=0}\p Z^cu||\p^{\le3}P_{\neq0}Z^du|\|_{L^2_{x,y}}\\
\ls\ve_1^3\w{t}^{2\ve_2-0.4}.
\end{equation*}
This, together with \eqref{NL:L2:pf4} and \eqref{NL:L2:pf5}, derives
\begin{equation}\label{NL:L2:pf6}
\sum_{|a|\le N+7}\|\w{t+|x|}Z^a\cC_2^i\|_{L^2_{x,y}}\ls\ve_1^3\w{t}^{3\ve_2-0.3}.
\end{equation}

\vskip 0.1 true cm
{\bf $\bullet$ Estimate of $Z^aG^i$}
\vskip 0.1 true cm

By \eqref{BA1}-\eqref{BA5}, one can easily obtains
\begin{equation}\label{NL:L2:pf7}
\sum_{|a|\le N+7}\|\w{t+|x|}Z^aG^i\|_{L^2_{x,y}}\ls\ve_1^4\w{t}^{\ve_2-1/2}.
\end{equation}

Then \eqref{NL:L2} follows from \eqref{NL:L2:pf3}, \eqref{NL:L2:pf6} and \eqref{NL:L2:pf7}.
\end{proof}

Next, we  improve the estimates \eqref{BA4} and \eqref{BA5}.
\begin{lemma}
Let $u(t,x,y)$ be the solution of \eqref{QWE} and suppose that \eqref{BA1}-\eqref{BA5} hold.
Then one has
\begin{equation}\label{pwKG:improv}
\begin{split}
&\sum_{|a|\le2N-8}|P_{\neq0}Z^au(t,x,y)|\ls(\ve+\ve_1^2)\w{t+|x|}^{\ve_2-1/2}(\w{t+|x|}^{-0.45}+\w{x}^{-1/2}\w{t-|x|}^{-0.4}),\\
&\sum_{|a|\le N+2}|P_{\neq0}Z^au(t,x,y)|\ls(\ve+\ve_1^2)\w{t+|x|}^{-1}.
\end{split}
\end{equation}
\end{lemma}
\begin{proof}
At first, it follows from Lemmas \ref{lem:pwKG}-\ref{lem:NL:L2} together with \eqref{initial:data}, \eqref{energy:V:pf1} and \eqref{NL:L2}
that
\begin{equation}\label{pwKG:improv1}
\begin{split}
&\sum_{|a|\le2N-8}|P_{\neq0}Z^aV|\ls(\ve+\ve_1^2)\w{t+|x|}^{-0.95},\\
&\sum_{|a|\le N+2}|P_{\neq0}Z^aV|\ls(\ve+\ve_1^2)\w{t+|x|}^{-1}.
\end{split}
\end{equation}

Next, we prove \eqref{pwKG:improv}. Note that
\eqref{normal:form} and \eqref{proj:property} lead to
\begin{equation}\label{pwKG:improv2}
\begin{split}
|P_{\neq0}Z^a(u-V)|&\ls\sum_{|b|+|c|\le|a|}\|\p^{\le1}P_{\neq0}Z^bu\|_{L_y^\infty}(\|\p^{\le1}P_{=0}Z^cu\|_{L_y^\infty}
+\|\p^{\le1}P_{\neq0}Z^cu\|_{L_y^\infty})\\
&\quad+\sum_{|b|+|c|+|d|\le|a|}\|\p^{\le2}P_{\neq0}Z^bu\|_{L_y^\infty}\|\p^{\le2}P_{\neq0}Z^cu\|_{L_y^\infty}
\|\p^{\le1}P_{=0}Z^du\|_{L_y^\infty}.
\end{split}
\end{equation}
From the first line in \eqref{pwKG:improv2}, one knows $|b|+|c|\le|a|\le2N-8$.
If $|c|\ge N-1$, then $|b|\le N-7$.
Thus, collecting \eqref{BA1}, \eqref{BA5}, \eqref{pw:high} and \eqref{pw0:high} yields
\begin{equation}\label{pwKG:improv3}
\sum_{\substack{|b|+|c|\le2N-8,\\|c|\ge N-1}}\|\p^{\le1}P_{\neq0}Z^bu\|_{L_y^\infty}(\|\p^{\le1}P_{=0}Z^cu\|_{L_y^\infty}+\|\p^{\le1}P_{\neq0}Z^cu\|_{L_y^\infty})
\ls\ve_1^2\w{t+|x|}^{2\ve_2-1}.
\end{equation}
If $|c|\le N-2$, then it can be deduced from \eqref{BA1}, \eqref{BA5}, \eqref{pw:high} and \eqref{pw0:low2} that
\begin{equation}\label{pwKG:improv4}
\begin{split}
&\sum_{\substack{|b|+|c|\le2N-8,\\|c|\le N-2}}\|\p^{\le1}P_{\neq0}Z^bu\|_{L_y^\infty}(\|\p^{\le1}P_{=0}Z^cu\|_{L_y^\infty}+\|\p^{\le1}P_{\neq0}Z^cu\|_{L_y^\infty})\\
&\ls\ve_1^2\w{t}^{\ve_2}\w{x}^{-1/2}\w{t+|x|}^{-1/2}\w{t-|x|}^{-0.4}.
\end{split}
\end{equation}
For the terms in the second line of \eqref{pwKG:improv2}, it follows from \eqref{BA1}, \eqref{BA2}, \eqref{BA5}, \eqref{pw:high}
and \eqref{pw0:high} that
\begin{equation}\label{pwKG:improv5}
\begin{split}
&\sum_{|b|+|c|+|d|\le2N-8}\|\p^{\le2}P_{\neq0}Z^bu\|_{L_y^\infty}\|\p^{\le2}P_{\neq0}Z^cu\|_{L_y^\infty}\|\p^{\le1}P_{=0}Z^du\|_{L_y^\infty}\\
&\ls\ve_1^3\w{t+|x|}^{2\ve_2-2}+\ve_1^3\w{t}^{\ve_2}\w{x}^{-1/2}\w{t+|x|}^{-3/2}\w{t-|x|}^{-0.3}\\
&\ls\ve_1^3\w{t+|x|}^{\ve_2-1.8}.
\end{split}
\end{equation}
Therefore, by \eqref{pwKG:improv1}-\eqref{pwKG:improv5}, the first inequality in \eqref{pwKG:improv} is proved.

At last, we show the second inequality in \eqref{pwKG:improv}.
By $|b|+|c|\le|a|\le N+2$, \eqref{BA2} and \eqref{BA4} ensure that
\begin{equation*}
\sum_{|b|+|c|\le N+2}\|\p^{\le1}P_{\neq0}Z^bu\|_{L_y^\infty}(\|\p^{\le1}P_{=0}Z^cu\|_{L_y^\infty}+\|\p^{\le1}P_{\neq0}Z^cu\|_{L_y^\infty})
\ls\ve_1^2\w{t+|x|}^{\ve_2-1.4}.
\end{equation*}
Together with \eqref{pwKG:improv1}, \eqref{pwKG:improv2} and \eqref{pwKG:improv5},
this derives the second inequality in \eqref{pwKG:improv}.
\end{proof}

\section{Energy estimate}\label{sect5}

\begin{lemma}\label{lem:energy}
Let $u$ be the solution of \eqref{QWE} and suppose that \eqref{BA1}-\eqref{BA5} hold.
Then
\begin{equation}\label{energy}
E^2_{2N}[u](t')\ls\ve^2+\ve_1\int_0^{t'}\w{t}^{-1}E^2_{2N}[u](t)dt.
\end{equation}
\end{lemma}
\begin{proof}
For any multi-index $a$ with $|a|\le2N$, multiplying \eqref{eqn:high} by $e^q\p_tZ^au^i$ with $q=q(|x|-t)$ and $q(s)=\int_{-\infty}^s\frac{dt}{\w{t}^{1.1}}$ yields that
\begin{equation}\label{energy1}
\begin{split}
&\frac12\sum_{i=1}^m\p_t(e^q|\p Z^au^i|^2)-\sum_{i=1}^m\sum_{\gamma=1}^3\p_{\gamma}(e^q\p_tZ^au^i\p_{\gamma}Z^au^i)
+\sum_{i=1}^m\frac{e^q}{2\w{t-|x}^{1.1}}(|\bar\p Z^au^i|^2+|\p_yZ^au^i|^2)\\
&=\sum_{i=1}^m e^q\p_tZ^au^i\Box Z^au^i
=e^qQ_{ijk}^{\alpha\beta\mu}\p^2_{\alpha\beta}Z^au^j\p_{\mu}u^k\p_tZ^au^i
+e^qQ_{ijkl}^{\alpha\beta\mu\nu}\p^2_{\alpha\beta}Z^au^j\p_{\mu}u^k\p_{\nu}u^l\p_tZ^au^i\\
&+e^q\p_tZ^au^i\sum_{\substack{b+c\le a,\\|b|<|a|}}Q_{ijk,abc}^{\alpha\beta\mu}\p^2_{\alpha\beta}Z^bu^j\p_{\mu}Z^cu^k
+e^q\p_tZ^au^i\sum_{\substack{b+c+d\le a,\\|b|<|a|}}Q_{ijkl,abcd}^{\alpha\beta\mu\nu}\p^2_{\alpha\beta}Z^bu^j\p_{\mu}Z^cu^k\p_{\nu}Z^du^l\\
&+e^q\p_tZ^au^i\sum_{b+c\le a}Q_{ijk,abc}^{\alpha\beta}\p_{\alpha}Z^bu^j\p_{\beta}Z^cu^k
+e^q\p_tZ^au^i\sum_{b+c+d\le a}C^{abcd}_{ijkl}Q_0(Z^bu^j,Z^cu^k)Z^du^l,
\end{split}
\end{equation}
where the summations over $\alpha,\beta,\mu,\nu=0,\cdots,3$ and $i,j,k,l=1,\cdots,m$ in \eqref{energy1} are omitted.
By the symmetric conditions \eqref{sym:condition}, for the terms on the second line of \eqref{energy1}, we have
\begin{equation}\label{energy2}
\begin{split}
&e^qQ_{ijk}^{\alpha\beta\mu}\p^2_{\alpha\beta}Z^au^j\p_{\mu}u^k\p_tZ^au^i
+e^qQ_{ijkl}^{\alpha\beta\mu\nu}\p^2_{\alpha\beta}Z^au^j\p_{\mu}u^k\p_{\nu}u^l\p_tZ^au^i\\
&=\p_\alpha(e^qQ_{ijk}^{\alpha\beta\mu}\p_{\beta}Z^au^j\p_{\mu}u^k\p_tZ^au^i
+e^qQ_{ijkl}^{\alpha\beta\mu\nu}\p_{\beta}Z^au^j\p_{\mu}u^k\p_{\nu}u^l\p_tZ^au^i)\\
&\quad-e^q\p_tZ^au^i\p_{\alpha}q(Q_{ijk}^{\alpha\beta\mu}\p_{\beta}Z^au^j\p_{\mu}u^k
+Q_{ijkl}^{\alpha\beta\mu\nu}\p_{\beta}Z^au^j\p_{\mu}u^k\p_{\nu}u^l)\\
&\quad-e^q\p_tZ^au^i[Q_{ijk}^{\alpha\beta\mu}\p_{\beta}Z^au^j\p^2_{\alpha\mu}u^k
+Q_{ijkl}^{\alpha\beta\mu\nu}\p_{\beta}Z^au^j(\p^2_{\alpha\mu}u^k\p_{\nu}u^l+\p_{\mu}u^k\p^2_{\alpha\nu}u^l)]\\
&\quad-e^q\p_t\p_{\alpha}Z^au^i(Q_{ijk}^{\alpha\beta\mu}\p_{\beta}Z^au^j\p_{\mu}u^k
+Q_{ijkl}^{\alpha\beta\mu\nu}\p_{\beta}Z^au^j\p_{\mu}u^k\p_{\nu}u^l)
\end{split}
\end{equation}
and
\begin{equation}\label{energy3}
\begin{split}
&-e^q\p_t\p_{\alpha}Z^au^i(Q_{ijk}^{\alpha\beta\mu}\p_{\beta}Z^au^j\p_{\mu}u^k
+Q_{ijkl}^{\alpha\beta\mu\nu}\p_{\beta}Z^au^j\p_{\mu}u^k\p_{\nu}u^l)\\
&=-\frac12\p_t[e^q\p_{\alpha}Z^au^i\p_{\beta}Z^au^j(Q_{ijk}^{\alpha\beta\mu}\p_{\mu}u^k+Q_{ijkl}^{\alpha\beta\mu\nu}\p_{\mu}u^k\p_{\nu}u^l)]\\
&\quad+\frac12e^q\p_tq\p_{\alpha}Z^au^i\p_{\beta}Z^au^j(Q_{ijk}^{\alpha\beta\mu}\p_{\mu}u^k+Q_{ijkl}^{\alpha\beta\mu\nu}\p_{\mu}u^k\p_{\nu}u^l)\\
&\quad+\frac12e^q\p_{\alpha}Z^au^i\p_{\beta}Z^au^j[Q_{ijk}^{\alpha\beta\mu}\p^2_{0\mu}u^k
+Q_{ijkl}^{\alpha\beta\mu\nu}(\p^2_{0\mu}u^k\p_{\nu}u^l+\p_{\mu}u^k\p^2_{0\nu}u^l)].
\end{split}
\end{equation}
By integrating \eqref{energy1}-\eqref{energy3}  over $[0,t']\times\R^2\times\T$, summing over all $|a|\le2N$,
and applying \eqref{BA:pw:full}, we arrive at
\begin{equation}\label{energy4}
\begin{split}
&\quad E^2_{2N}[u](t')+\sum_{|a|\le2N}\int_0^{t'}\iint_{\R^2\times\T}\frac{|\bar\p Z^au|^2+|\p_yZ^au|^2}{\w{t-|x|}^{1.1}}dxdydt\\
&\ls E^2_{2N}[u](0)+\ve_1 E^2_{2N}[u](t')+\sum_{|a|\le2N}\int_0^{t'}\iint_{\R^2\times\T}I(t,x,y)dxdydt\\
&+\sum_{\substack{|b|+|c|\le|a|\le2N,\\|b|\le2N-1}}\int_0^{t'}\iint_{\R^2\times\T}
|\p_tZ^au^i|\Big|\sum_{\alpha,\beta,\mu=0}^3Q_{ijk,abc}^{\alpha\beta\mu}\p^2_{\alpha\beta}Z^bu^j\p_{\mu}Z^cu^k\Big|dxdydt\\
&+\sum_{\substack{|b|+|c|+|d|\le|a|\le2N,\\|b|\le2N-1}}\int_0^{t'}\iint_{\R^2\times\T}
|\p_tZ^au^i|\Big|\sum_{\alpha,\beta,\mu,\nu=0}^3Q_{ijkl,abcd}^{\alpha\beta\mu\nu}\p^2_{\alpha\beta}Z^bu^j\p_{\mu}Z^cu^k\p_{\nu}Z^du^l\Big|dxdydt\\
&+\sum_{|b|+|c|\le|a|\le2N}\int_0^{t'}\iint_{\R^2\times\T}|\p_tZ^au^i|
\Big|\sum_{\alpha,\beta=0}^3Q_{ijk,abc}^{\alpha\beta}\p_{\alpha}Z^bu^j\p_{\beta}Z^cu^k\Big|dxdydt\\
&+\sum_{|b|+|c|+|d|\le|a|\le2N}\int_0^{t'}\iint_{\R^2\times\T}|\p_tZ^au^iQ_0(Z^bu^j,Z^cu^k)Z^du^l|dxdydt,
\end{split}
\end{equation}
where
\begin{equation}\label{energy5}
\begin{split}
&\quad I(t,x,y):=\\
&|\p_tZ^au^i|\Big(\Big|\sum_{\alpha,\beta,\mu=0}^3Q_{ijk}^{\alpha\beta\mu}\p_{\alpha}q\p_{\beta}Z^au^j\p_{\mu}u^k\Big|
+\Big|\sum_{\alpha,\beta,\mu,\nu=0}^3Q_{ijkl}^{\alpha\beta\mu\nu}\p_{\alpha}q\p_{\beta}Z^au^j\p_{\mu}u^k\p_{\nu}u^l\Big|\Big)\\
&+|\p_tZ^au^i|\Big|\sum_{\alpha,\beta,\mu=0}^3Q_{ijk}^{\alpha\beta\mu}\p_{\beta}Z^au^j\p^2_{\alpha\mu}u^k\Big|
+\Big|\sum_{\alpha,\beta,\mu=0}^3Q_{ijk}^{\alpha\beta\mu}\p_{\alpha}Z^au^i\p_{\beta}Z^au^j\p^2_{0\mu}u^k\Big|\\
&+|\p_tZ^au^i|\Big|\sum_{\alpha,\beta,\mu,\nu=0}^3Q_{ijkl}^{\alpha\beta\mu\nu}\p_{\beta}Z^au^j(\p^2_{\alpha\mu}u^k\p_{\nu}u^l
+\p_{\mu}u^k\p^2_{\alpha\nu}u^l)\Big|\\
&+|\p_tq|\Big(\Big|\sum_{\alpha,\beta,\mu=0}^3Q_{ijk}^{\alpha\beta\mu}\p_{\alpha}Z^au^i\p_{\beta}Z^au^j\p_{\mu}u^k\Big|
+\Big|\sum_{\alpha,\beta,\mu,\nu=0}^3Q_{ijkl}^{\alpha\beta\mu\nu}\p_{\alpha}Z^au^i\p_{\beta}Z^au^j\p_{\mu}u^k\p_{\nu}u^l\Big|\Big)\\
&+\Big|\sum_{\alpha,\beta,\mu,\nu=0}^3Q_{ijkl}^{\alpha\beta\mu\nu}\p_{\alpha}Z^au^i\p_{\beta}Z^au^j(\p^2_{0\mu}u^k\p_{\nu}u^l
+\p_{\mu}u^k\p^2_{0\nu}u^l)\Big|.
\end{split}
\end{equation}
Note that $|b|\le N$ or $|c|\le N$ holds in the third and fifth lines of \eqref{energy4}.

We now treat the cases of $|b|\le N$.
Note that by $P_{=0}\p_y=0$, one has
\begin{equation*}\label{energy6}
\begin{split}
&\sum_{\alpha,\beta,\mu=0}^3Q_{ijk,abc}^{\alpha\beta\mu}\p^2_{\alpha\beta}Z^bu^j\p_{\mu}Z^cu^k
=\sum_{\alpha,\beta,\mu=0}^3Q_{ijk,abc}^{\alpha\beta\mu}\p_{\mu}Z^cu^k(P_{=0}\p^2_{\alpha\beta}Z^bu^j+P_{\neq0}\p^2_{\alpha\beta}Z^bu^j)\\
\end{split}
\end{equation*}

\begin{equation}\label{energy6}
\begin{split}
&=\sum_{\alpha,\beta=0}^2Q_{ijk,abc}^{\alpha\beta3}\p_yZ^cu^kP_{=0}\p^2_{\alpha\beta}Z^bu^j
+\sum_{\alpha,\beta,\mu=0}^2Q_{ijk,abc}^{\alpha\beta\mu}\p_{\mu}Z^cu^kP_{=0}\p^2_{\alpha\beta}Z^bu^j\\
&\quad+\sum_{\alpha,\beta,\mu=0}^3Q_{ijk,abc}^{\alpha\beta\mu}\p_{\mu}Z^cu^kP_{\neq0}\p^2_{\alpha\beta}Z^bu^j.
\end{split}
\end{equation}
For the terms in the third line of \eqref{energy6}, in terms of \eqref{BA5}, we have
\begin{equation}\label{energy7}
|\p_{\mu}Z^cu^kP_{\neq0}\p^2_{\alpha\beta}Z^bu^j|\ls\ve_1\w{t}^{-1}|\p Z^cu|.
\end{equation}
For the terms in the second line of \eqref{energy6}, when $|x|\le\w{t}/2$,
one has from \eqref{BA3} that
\begin{equation}\label{energy8}
\begin{split}
|\p_yZ^cu^kP_{=0}\p^2_{\alpha\beta}Z^bu^j|+|\p_{\mu}Z^cu^kP_{=0}\p^2_{\alpha\beta}Z^bu^j|\ls\ve_1\w{t}^{-1.3}|\p Z^cu|.
\end{split}
\end{equation}
When $|x|\ge\w{t}/2$, it is concluded from \eqref{null:structure}, \eqref{BA3}, \eqref{pw0:good:low} and Young's inequality that
\begin{equation}\label{energy9}
\begin{split}
&\quad\;|\p_tZ^au^i||\p_yZ^cu^kP_{=0}\p^2_{\alpha\beta}Z^bu^j|
+|\p_tZ^au^i|\Big|\sum_{\alpha,\beta,\mu=0}^2Q_{ijk,abc}^{\alpha\beta\mu}\p_{\mu}Z^cu^jP_{=0}\p^2_{\alpha\beta}Z^bu^k\Big|\\
&\ls|\p Z^au||P_{=0}\p^2Z^bu|(|\bar\p Z^cu|+|\p_yZ^cu|)+|\p Z^au||\p Z^cu||P_{=0}\bar\p\p Z^bu|\\
&\ls\ve_1\w{t}^{-1/2}\w{t-|x|}^{-1.3}|\p Z^au|(|\bar\p Z^cu|+|\p_yZ^cu|)+\ve_1\w{t}^{-3/2}|\p Z^au||\p Z^cu|\\
&\ls\ve_1\w{t-|x|}^{-2.6}(|\bar\p Z^cu|^2+|\p_yZ^cu|^2)+\ve_1\w{t}^{-1}(|\p Z^au|^2+|\p Z^au||\p Z^cu|).
\end{split}
\end{equation}
Substituting \eqref{energy7}-\eqref{energy9} into \eqref{energy6} yields
\begin{equation}\label{energy10}
\begin{split}
&\quad\sum_{|b|\le N}|\p_tZ^au^i|\Big|\sum_{\alpha,\beta,\mu=0}^3Q_{ijk,abc}^{\alpha\beta\mu}\p^2_{\alpha\beta}Z^bu^j\p_{\mu}Z^cu^k\Big|\\
&\ls\ve_1\w{t-|x|}^{-2.6}(|\bar\p Z^cu|^2+|\p_yZ^cu|^2)+\ve_1\w{t}^{-1}(|\p Z^au|^2+|\p Z^au||\p Z^cu|).
\end{split}
\end{equation}

For the cases of $|c|\le N$,
analogously to \eqref{energy6}, we have
\begin{equation*}
\begin{split}
&\sum_{\alpha,\beta,\mu=0}^3Q_{ijk,abc}^{\alpha\beta\mu}\p^2_{\alpha\beta}Z^bu^j\p_{\mu}Z^cu^k
=\sum_{\alpha,\beta,\mu=0}^3Q_{ijk,abc}^{\alpha\beta\mu}\p^2_{\alpha\beta}Z^bu^j(P_{=0}\p_{\mu}Z^cu^k+P_{\neq0}\p_{\mu}Z^cu^k)\\
&=\sum_{\substack{\alpha,\beta,\mu=0,\\\alpha=3~\mathrm{or}~\beta=3}}^3Q_{ijk,abc}^{\alpha\beta\mu}\p^2_{\alpha\beta}Z^bu^jP_{=0}\p_{\mu}Z^cu^k
+\sum_{\alpha,\beta,\mu=0}^2Q_{ijk,abc}^{\alpha\beta\mu}\p^2_{\alpha\beta}Z^bu^jP_{=0}\p_{\mu}Z^cu^k\\
&\quad+\sum_{\alpha,\beta,\mu=0}^3Q_{ijk,abc}^{\alpha\beta\mu}\p^2_{\alpha\beta}Z^bu^jP_{\neq0}\p_{\mu}Z^cu^k.
\end{split}
\end{equation*}
Then the following estimate similar to \eqref{energy10} can be obtained
\begin{equation}\label{energy11}
\begin{split}
&\quad\sum_{|c|\le N}|\p_tZ^au^i|\Big|\sum_{\alpha,\beta,\mu=0}^3Q_{ijk,abc}^{\alpha\beta\mu}\p^2_{\alpha\beta}Z^bu^j\p_{\mu}Z^cu^k\Big|\\
&\ls\ve_1\w{t-|x|}^{-2.6}(|\bar\p\p Z^bu|^2+|\p_y\p Z^bu|^2)+\ve_1\w{t}^{-1}(|\p Z^au|^2+|\p Z^au||\p^2Z^bu|).
\end{split}
\end{equation}
Collecting \eqref{energy10} and \eqref{energy11} yields
\begin{equation}\label{energy12}
\begin{split}
&\sum_{\substack{|b|+|c|\le|a|\le2N,\\|b|\le2N-1}}\int_0^{t'}\iint_{\R^2\times\T}
|\p_tZ^a\vp|\Big|\sum_{\alpha,\beta,\mu=0}^3Q_{abc}^{\alpha\beta\mu}\p^2_{\alpha\beta}Z^b\vp\p_{\mu}Z^c\vp\Big|dxdydt\\
&\ls\ve_1\sum_{|c|\le2N}\int_0^{t'}\iint_{\R^2\times\T}\frac{|\bar\p Z^c\vp|^2+|\p_yZ^c\vp|^2}{\w{t-|x|}^{1.1}}dxdydt
+\ve_1\int_0^{t'}\w{t}^{-1}E^2_{2N}[u](t)dt.
\end{split}
\end{equation}
Analogously, we can get the same estimate as in \eqref{energy12} for the terms in the fourth and fifth lines of \eqref{energy4}.

Next, we start to  treat $\ds\sum_{|a|\le2N}\int_0^{t'}\iint_{\R^2\times\T}I(t,x,y)dxdydt$
with $I(t,x,y)$ defined by \eqref{energy5}.

For the first term in $I(t,x,y)$, similarly to \eqref{energy6}, from $\p_yq=0$, we have
\begin{equation*}
\begin{split}
&\sum_{\alpha,\beta,\mu=0}^3Q_{ijk}^{\alpha\beta\mu}\p_{\alpha}q\p_{\beta}Z^au^j\p_{\mu}u^k
=\sum_{\alpha,\beta,\mu=0}^3Q_{ijk}^{\alpha\beta\mu}\p_{\alpha}q\p_{\beta}Z^au^j(P_{=0}\p_{\mu}u^k+P_{\neq0}\p_{\mu}u^k)\\
&=\sum_{\alpha,\mu=0}^2Q_{ijk}^{\alpha\beta\mu}\p_{\alpha}q\p_yZ^au^jP_{=0}\p_{\mu}u^k
+\sum_{\alpha,\beta,\mu=0}^2Q_{ijk}^{\alpha\beta\mu}\p_{\alpha}q\p_{\beta}Z^au^jP_{=0}\p_{\mu}u^k\\
&\quad+\sum_{\alpha,\beta,\mu=0}^3Q_{ijk}^{\alpha\beta\mu}\p_{\alpha}q\p_{\beta}Z^au^jP_{\neq0}\p_{\mu}u^k.
\end{split}
\end{equation*}
In addition, by $\bar\p q=0$ and \eqref{null:structure}, one can arrive at
\begin{equation*}
\begin{split}
\Big|\sum_{\alpha,\beta,\mu=0}^2Q_{ijk}^{\alpha\beta\mu}\p_{\alpha}q\p_{\beta}Z^au^jP_{=0}\p_{\mu}u^k\Big|
\ls|\bar\p Z^au||P_{=0}\p u|+|\p Z^au||P_{=0}\bar\p u|.
\end{split}
\end{equation*}
Thus the estimate of $\sum_{|a|\le2N}\int_0^{t'}\iint_{\R^2\times\T}I(t,x,y)dxdydt$ is the same as for \eqref{energy12}.

At last, we estimate the last line in \eqref{energy4}.
When $|d|\le N$, it follows from \eqref{BA:pw:full} that
\begin{equation}\label{energy13}
\sum_{\substack{|b|+|c|+|d|\le2N,\\|d|\le N}}\int_0^{t'}\iint_{\R^2\times\T}|\p_tZ^au^iQ_0(Z^bu^j,Z^cu^k)Z^du^l|dxdydt
\ls\ve_1^2\int_0^{t'}\w{t}^{-1}E^2_{2N}[u](t)dt.
\end{equation}
When $|d|\ge N$, one can see that $|b|+|c|\le N$ and $Z^du^l=P_{\neq0}Z^du^l+P_{=0}Z^du^l$.
By  \eqref{Poincare:ineq} and \eqref{BA:pw:full}, we have
\begin{equation}\label{energy14}
\sum_{\substack{|b|+|c|+|d|\le2N,\\|d|\ge N}}\int_0^{t'}\iint_{\R^2\times\T}|\p_tZ^au^iQ_0(Z^bu^j,Z^cu^k)P_{\neq0}Z^du^l|dxdydt
\ls\ve_1^2\int_0^{t'}\w{t}^{-1}E^2_{2N}[u](t)dt.
\end{equation}
On the other hand, it is derived from \eqref{energy:V:pf5} that
\begin{equation}\label{energy15}
|P_{=0}Z^du|\ls\w{t+|x|}\sum_{|d'|\le|d|-1\le2N-1}|P_{=0}\p Z^{d'}u|.
\end{equation}
In addition, by \eqref{proj:property}, we have
\begin{equation}\label{energy16}
\begin{split}
Q_0(Z^bu^j,Z^cu^k)&=Q_0(P_{\neq0}Z^bu^j,P_{\neq0}Z^cu^k)+Q_0(P_{=0}Z^bu^j,P_{=0}Z^cu^k)\\
&\quad+Q_0(P_{=0}Z^bu^j,P_{\neq0}Z^cu^k)+Q_0(P_{\neq0}Z^bu^j,P_{=0}Z^cu^k).
\end{split}
\end{equation}
For the first line in \eqref{energy16}, it follows from \eqref{null:structure}, \eqref{BA3}, \eqref{BA5} and \eqref{pw0:good:low} that
\begin{equation}\label{energy17}
\begin{split}
&\quad\;|Q_0(P_{\neq0}Z^bu^j,P_{\neq0}Z^cu^k)|+|Q_0(P_{=0}Z^bu^j,P_{=0}Z^cu^k)|\\
&\ls|P_{\neq0}\p Z^bu||P_{\neq0}\p Z^cu|+|P_{=0}\bar\p Z^bu||P_{=0}\p Z^cu|+|P_{=0}\p Z^bu||P_{=0}\bar\p Z^cu|\\
&\ls\ve_1^2\w{t+|x|}^{-2}.
\end{split}
\end{equation}
For the second line in \eqref{energy16}, by \eqref{null:structure}, \eqref{BA3}, \eqref{BA5}, \eqref{pw0:good:low} and
\eqref{pwKG:good},
one can obtain
\begin{equation}\label{energy18}
\begin{split}
&\quad\;|Q_0(P_{=0}Z^bu^j,P_{\neq0}Z^cu^k)|+|Q_0(P_{\neq0}Z^bu^j,P_{=0}Z^cu^k)|\\
&\ls|P_{=0}\bar\p Z^bu||P_{\neq0}\p Z^cu|+|P_{=0}\p Z^bu||P_{\neq0}\bar\p Z^cu|\\
&\quad+|P_{\neq0}\bar\p Z^bu||P_{=0}\p Z^cu|+|P_{\neq0}\p Z^bu||P_{=0}\bar\p Z^cu|\\
&\ls\ve_1^2\w{t+|x|}^{-2.3}.
\end{split}
\end{equation}
Collecting \eqref{energy15}-\eqref{energy18} yields
\begin{equation}\label{energy19}
\sum_{\substack{|b|+|c|+|d|\le2N,\\|d|\ge N}}\int_0^{t'}\iint_{\R^2\times\T}|\p_tZ^au^iQ_0(Z^bu^j,Z^cu^k)P_{=0}Z^du^l|dxdydt
\ls\ve_1^2\int_0^{t'}\w{t}^{-1}E^2_{2N}[u](t)dt.
\end{equation}
Therefore, it follows from \eqref{initial:data}, \eqref{energy4}, \eqref{energy12}-\eqref{energy14},
\eqref{energy19} and the smallness of $\ve_1$ that \eqref{energy} holds.
\end{proof}

\section{Proof of Theorem \ref{thm1}}
\begin{proof}
It follows from \eqref{pw0:improv}, \eqref{pwKG:improv},
\eqref{energy} and Gronwall's inequality that there are two constants $C_1,C_2>1$ such that
\begin{equation*}
E_{2N}[u](t)\le C_1(\ve+\ve_1^2)(1+t)^{C_2\ve_1}
\end{equation*}
and
\begin{equation*}
\begin{split}
&\sum_{|a|\le N+7}|P_{=0}Z^au(t,x,y)|\le C_1(\ve+\ve_1^2)\w{t+|x|}^{-1/2}\w{t-|x|}^{-0.3},\\
&\sum_{|a|\le N+6}|P_{=0}\p Z^au(t,x,y)|\le C_1(\ve+\ve_1^2)\w{x}^{-1/2}\w{t-|x|}^{-1.3},\\
&\sum_{|a|\le2N-8}|P_{\neq0}Z^au(t,x,y)|\le C_1(\ve+\ve_1^2)\w{t+|x|}^{\ve_2-1/2}(\w{t+|x|}^{-0.45}+\w{x}^{-1/2}\w{t-|x|}^{-0.4}),\\
&\sum_{|a|\le N+2}|P_{\neq0}Z^au(t,x,y)|\le C_1(\ve+\ve_1^2)\w{t+|x|}^{-1}.
\end{split}
\end{equation*}
Let $\ve_1=4C_1\ve$ and $\ve_0=\min\{\frac{1}{16C_1^2+1},\frac{1}{400C_1C_2+1}\}$.
Then  \eqref{BA1}-\eqref{BA5} are improved to
\begin{equation*}
E_{2N}[u](t)\le\frac12\ve_1(1+t)^{\ve_2}
\end{equation*}
and
\begin{equation*}
\begin{split}
&\sum_{|a|\le N+7}|P_{=0}Z^au(t,x,y)|\le\frac12\ve_1\w{t+|x|}^{-1/2}\w{t-|x|}^{-0.3},\\
&\sum_{|a|\le N+6}|P_{=0}\p Z^au(t,x,y)|\le\frac12\ve_1\w{x}^{-1/2}\w{t-|x|}^{-1.3},\\
&\sum_{|a|\le2N-8}|P_{\neq0}Z^au(t,x,y)|\le\frac12\ve_1\w{t+|x|}^{\ve_2-1/2}(\w{t+|x|}^{-0.45}+\w{x}^{-1/2}\w{t-|x|}^{-0.4}),\\
&\sum_{|a|\le N+2}|P_{\neq0}Z^au(t,x,y)|\le\frac12\ve_1\w{t+|x|}^{-1}.
\end{split}
\end{equation*}
This, together with the local existence of classical solution to \eqref{QWE}, yields that \eqref{QWE} with \eqref{part:null:condtion}
and \eqref{sym:condition} admits a unique global solution $u\in\bigcap\limits_{j=0}^{2N+1}C^{j}([0,\infty), H^{2N+1-j}(\R^2\times\T)))$.
Moreover, \eqref{thm1:decay} can be achieved by \eqref{proj:property}, \eqref{BA3}, \eqref{BA5} and \eqref{pw0:low2}.

At last, we prove \eqref{thm1:scater}.
First, we establish
\begin{equation}\label{scater:pf1}
E_{0}[u-V](t)\ls\ve^2\w{t}^{-1/2}.
\end{equation}
It follows from \eqref{energy:ineq} with \eqref{initial:data}, \eqref{normal:form} and \eqref{NL:L2} that
\begin{equation*}
\begin{split}
\big(E_{N+6}[V](t)\big)^2
&\ls\big(E_{N+6}[V](0)\big)^2+\ve^3\int_0^t(1+s)^{-1.2}ds\sup_{s\in[0,t]}E_{N+6}[V](s)\\
&\ls(\ve+\ve^3)\sup_{s\in[0,t]}E_{N+6}[V](s),
\end{split}
\end{equation*}
which implies
\begin{equation}\label{scater:pf2}
E_{N+6}[V](t)\ls\ve.
\end{equation}
On the other hand, by \eqref{normal:form}, one has
\begin{equation}\label{scater:pf3}
\begin{split}
E_{N+6}[u-V](t)&\ls\sum_{|b|+|c|\le N+6}\||\p^{\le1}\p Z^bu||\p^{\le1}Z^cu|\|_{L^2_{x,y}}\\
&+\sum_{|b|+|c|+|d|\le N+6}\||\p^{\le3}P_{\neq0}Z^bu||\p^{\le3}P_{\neq0}Z^cu||P_{=0}\p^{\le1}Z^du|\|_{L^2_{x,y}}.
\end{split}
\end{equation}
This, together with \eqref{proj:property}, \eqref{Poincare:ineq}, \eqref{BA1}, \eqref{BA2}, \eqref{BA5}, \eqref{BA:pw:full}
and \eqref{scater:pf2}, yields
\begin{equation}\label{scater:pf4}
E_{N+6}[u](t)\ls\ve+\ve\w{t}^{-1/2}(E_{N+7}[u](t)+\ve E_{N+8}[u](t))\ls\ve+\ve\w{t}^{-0.49}\ls\ve.
\end{equation}
Analogously to \eqref{scater:pf3}, we can obtain
\begin{equation*}
\begin{split}
E_{0}[u-V](t)&\ls\||\p^{\le1}\p u||\p^{\le1}u|+|\p^{\le3}P_{\neq0}u||\p^{\le3}P_{\neq0}u||P_{=0}\p^{\le1}u|\|_{L^2_{x,y}}\\
&\ls\ve\w{t}^{-1/2}E_{3}[u](t)\ls\ve^2\w{t}^{-1/2},
\end{split}
\end{equation*}
where \eqref{proj:property}, \eqref{Poincare:ineq}, \eqref{BA2}, \eqref{BA5}, \eqref{BA:pw:full} and \eqref{scater:pf4}
have been used.
Thus, \eqref{scater:pf1} is proved.

Next, we construct $u^\infty$ in \eqref{thm1:scater}.
Denote $\ds\Lambda:=\Big(-\sum_{j=1}^3\p^2_j\Big)^\frac12$ and $\cV:=(\p_t+i\Lambda)V$.
Then one has
\begin{equation*}
(\p_t-i\Lambda)\cV=\Box V.
\end{equation*}
Together with Duhamel's principle, this leads to
\begin{equation}\label{scater:pf5}
\cV(t)=e^{it\Lambda}\cV(0)+\int_0^te^{i(t-s)\Lambda}\Box V(s)ds.
\end{equation}
Set
\begin{equation}\label{scater:pf6}
\begin{split}
&\cU_{(0)}^\infty:=\cV(0)+\int_0^\infty e^{-is\Lambda}\Box V(s)ds,\\
&u^\infty_{(0)}:=\Lambda^{-1}\mathrm{Im}~\cU_{(0)}^\infty,\quad u^\infty_{(1)}:=\mathrm{Re}~\cU_{(0)}^\infty.
\end{split}
\end{equation}
It follows from \eqref{initial:data}, \eqref{NL:L2}, \eqref{scater:pf6}, the Minkowski inequality and the unitary of $e^{-is\Lambda}$ that
\begin{equation*}
\begin{split}
\|\cU_{(0)}^\infty\|_{L^2_{x,y}}&\ls\|\cV(0)\|_{L^2_{x,y}}+\int_0^\infty\|e^{-is\Lambda}\Box V(s)\|_{L^2_{x,y}}ds\\
&\ls\|\cV(0)\|_{L^2_{x,y}}+\int_0^\infty\|\Box V(s)\|_{L^2_{x,y}}ds\\
&\ls\ve,
\end{split}
\end{equation*}
which means $(u_{(0)}^\infty,u_{(1)}^\infty)\in\dot H^1(\R^2\times\T)\times L^2(\R^2\times\T)$.
On the other hand, $u^\infty(t)=\Lambda^{-1}\mathrm{Im}e^{it\Lambda}\cU_{(0)}^\infty$ is the solution
of $\Box u^\infty=0$ with the initial data $(u_{(0)}^\infty,u_{(1)}^\infty)$ at $t=0$.
Thus, \eqref{scater:pf5} and \eqref{scater:pf6} imply
\begin{equation}\label{scater:pf7}
\begin{split}
E_0[u^\infty-V](t)&\ls\|e^{it\Lambda}\cU_{(0)}^\infty-\cV(t)\|_{L^2_{x,y}}
\ls\int_t^\infty\|e^{i(t-s)\Lambda}\Box V(s)\|_{L^2_{x,y}}ds\\
&\ls\int_t^\infty\|\Box V(s)\|_{L^2_{x,y}}ds,
\end{split}
\end{equation}
where the fact of $V=\Lambda^{-1}\mathrm{Im}\cV$ is used.

Next, we estimate $\|\Box V\|_{L^2_{x,y}}$.
Applying \eqref{BA2}, \eqref{BA5}, \eqref{pw0:good:low}, \eqref{scater:pf4} to \eqref{energy:V:pf3}, \eqref{pw0:improv4}, \eqref{pw0:improv5} yields
\begin{equation}\label{scater:pf8}
\|\cC_1^i+\cC_2^i\|_{L^2_{x,y}}\ls\ve^3\w{t}^{-3/2}.
\end{equation}
On the other hand, thanks to \eqref{proj:property}, \eqref{Poincare:ineq}, \eqref{BA2}, \eqref{BA5} and \eqref{scater:pf4}, we have
\begin{equation*}
\|G^i\|_{L^2_{x,y}}\ls\ve^4\w{t}^{-3/2}.
\end{equation*}
This, together with \eqref{scater:pf7} and \eqref{scater:pf8}, gives
\begin{equation}\label{scater:pf9}
E_0[u^\infty-V](t)\ls\int_t^\infty\ve^3\w{s}^{-3/2}ds\ls\ve^3\w{t}^{-1/2}.
\end{equation}
Therefore, \eqref{thm1:scater} can be achieved by \eqref{scater:pf1} and \eqref{scater:pf9}.
\end{proof}

\appendix
\setcounter{equation}{1}

\section{Derivations of \eqref{QWE2} and \eqref{QWE3-1}}
\begin{proof}[Proofs of \eqref{QWE2} and \eqref{QWE3-1}]
At first, it holds that
\begin{equation}\label{App:A1}
\begin{split}
\Box(fg)&=g\Box f+f\Box g+2Q_0(f,g),\\
\Box(fgh)&=gh\Box f+fh\Box g+fg\Box h+2fQ_0(g,h)+2gQ_0(f,h)+2hQ_0(f,g).
\end{split}
\end{equation}
Then it follows from \eqref{App:A1} that
\begin{equation*}
\Box(\p^au^j\p^bu^k)=2Q_0(\p^au^j,\p^bu^k)+\p^a\Box u^j\p^bu^k+\p^au^j\p^b\Box u^k.
\end{equation*}
This, together with \eqref{QWE} and direct computation,  yields
\begin{equation*}
\begin{split}
&\Box\tilde V^i=\Box u^i-\frac12\sum_{j,k=1}^m\sum_{|a|+|b|\le1}C^{ab}_{ijk}(2Q_0(\p^au^j,\p^bu^k)+\p^a\Box u^j\p^bu^k+\p^au^j\p^b\Box u^k)\\
&=\sum_{j,k,l=1}^mC_{ijkl}Q_0(u^j,u^k)u^l+\sum_{j,k,l=1}^m\sum_{\alpha,\beta,\mu,\nu=0}^3Q_{ijkl}^{\alpha\beta\mu\nu}\p^2_{\alpha\beta}u^j
\p_{\mu}u^k\p_{\nu}u^l\\
&-\frac12\sum_{|a|+|b|\le1}\sum_{j,k=1}^mC^{ab}_{ijk}\Big\{C_{jj'k'l'}\p^a(Q_0(u^{j'},u^{k'})u^{l'})\p^bu^k+C_{kj'k'l'}\p^au^j
\p^b(Q_0(u^{j'},u^{k'})u^{l'})\\
&\quad+\sum_{|a'|+|b'|\le1}[C^{a'b'}_{jj'k'}\p^aQ_0(\p^{a'}u^{j'},\p^{b'}u^{k'})\p^bu^k
+C^{a'b'}_{kj'k'}\p^au^j\p^bQ_0(\p^{a'}u^{j'},\p^{b'}u^{k'})]\\
&\quad+Q_{jj'k'l'}^{\alpha\beta\mu\nu}\p^a(\p^2_{\alpha\beta}u^{j'}\p_{\mu}u^{k'}\p_{\nu}u^{l'})\p^bu^k
+Q_{kj'k'l'}^{\alpha\beta\mu\nu}\p^au^j\p^b(\p^2_{\alpha\beta}u^{j'}\p_{\mu}u^{k'}\p_{\nu}u^{l'})\Big\}.
\end{split}
\end{equation*}
Therefore, we have
\begin{equation}\label{App:A3}
\begin{split}
&\Box\tilde V^i=\sum_{j,k,l=1}^m\sum_{|a|,|b|\le2}C^{ab}_{ijkl}Q_0(\p^au^j,\p^bu^k)u^l
+\sum_{j,k,l=1}^m\sum_{\alpha,\beta,\mu,\nu=0}^3\tilde Q_{ijkl}^{\alpha\beta\mu\nu}\p^2_{\alpha\beta}u^j\p_{\mu}u^k\p_{\nu}u^l\\
&\quad+\sum_{j,k,l=1}^m\sum_{\alpha,\beta,\mu=0}^3\tilde Q_{ijkl}^{\alpha\beta\mu}\p_{\alpha}u^j\p_{\beta}u^k\p_{\mu}u^l
+G_4^i(\p^{\le3}u),\\
&G_4^i(\p^{\le3}u):=G_{4,0}^i(\p u,\p^2u,\p^3u)+\sum_{j=1}^mu^jG_{4,1}^{ij}(\p u,\p^2u,\p^3u)
+\sum_{j,k=1}^mu^ju^kG_{4,2}^{ijk}(\p u,\p^2u,\p^3u),
\end{split}
\end{equation}
where $G_{4,\ell}^{\ast\ast\ast}(\p u,\p^2u,\p^3u)=O((|\p u|+|\p^2u|+|\p^3u|)^{4-\ell})$
for $\ell=0,1,2$,  and $\tilde Q_{ijkl}^{\alpha\beta\mu\nu}$, $\tilde Q_{ijkl}^{\alpha\beta\mu}$ are
some suitable constants which satisfy such partial null conditions
\begin{equation}\label{App:A4}
\sum_{\alpha,\beta,\mu,\nu=0}^2\tilde Q_{ijkl}^{\alpha\beta\mu\nu}\xi_\alpha\xi_\beta\xi_\mu\xi_\nu\equiv0,\quad
\sum_{\alpha,\beta,\mu=0}^2\tilde Q_{ijkl}^{\alpha\beta\mu}\xi_\alpha\xi_\beta\xi_\mu\equiv0,\quad\forall(\xi_0,\xi_1,\xi_2)\in\{\pm1\}\times\SS.
\end{equation}
Thus, the proof of \eqref{QWE2} is finished.

In addition, according to the definition \eqref{normal:form} and \eqref{App:A1}, one can find that
\begin{equation*}
\begin{split}
&\Box V^i=\Box\tilde V^i-\frac12\sum_{j,k,l=1}^m\sum_{|a|,|b|\le2}C^{ab}_{ijkl}\Big\{2Q_0(P_{\neq0}\p^au^j,P_{\neq0}\p^bu^k)P_{=0}u^l\\
&+2Q_0(P_{\neq0}\p^au^j,P_{=0}u^l)P_{\neq0}\p^bu^k+2Q_0(P_{=0}u^l,P_{\neq0}\p^bu^k)P_{\neq0}\p^au^j\\
&+P_{\neq0}\p^a\Box u^jP_{\neq0}\p^bu^kP_{=0}u^l+P_{\neq0}\p^au^jP_{\neq0}\p^b\Box u^kP_{=0}u^l+P_{\neq0}\p^au^jP_{\neq0}\p^bu^kP_{=0}\Box u^l\Big\}.
\end{split}
\end{equation*}
This, together with \eqref{App:A3} and Lemma \ref{proj:property}, leads to
\begin{equation}\label{App:A5}
\begin{split}
&\Box V^i=\sum_{j,k,l=1}^m\sum_{\alpha,\beta,\mu,\nu=0}^3\tilde Q_{ijkl}^{\alpha\beta\mu\nu}\p^2_{\alpha\beta}u^j\p_{\mu}u^k\p_{\nu}u^l
+\sum_{j,k,l=1}^m\sum_{\alpha,\beta,\mu=0}^3\tilde Q_{ijkl}^{\alpha\beta\mu}\p_{\alpha}u^j\p_{\beta}u^k\p_{\mu}u^l\\
&+\sum_{j,k,l=1}^m\sum_{|a|,|b|\le2}C^{ab}_{ijkl}\Big\{Q_0(\p^au^j,\p^bu^k)P_{\neq0}u^l\\
&+P_{=0}u^l[Q_0(P_{=0}\p^au^j,P_{=0}\p^bu^k)+Q_0(P_{\neq0}\p^au^j,P_{=0}\p^bu^k)+Q_0(P_{=0}\p^au^j,P_{\neq0}\p^bu^k)]\\
&-Q_0(P_{\neq0}\p^au^j,P_{=0}u^l)P_{\neq0}\p^bu^k-Q_0(P_{=0}u^l,P_{\neq0}\p^bu^k)P_{\neq0}\p^au^j\Big\}+F_4^i(\p^{\le3}u)\\
&=:\cC_1^i+\cC_2^i+G^i,\\
&G^i:=G_4^i(\p^{\le3}u)+\tilde G^i(\p^{\le4}u),\\
&\tilde G^i(\p^{\le4}u):=-\frac12\sum_{j,k,l=1}^m\sum_{|a|,|b|\le2}C^{ab}_{ijkl}\Big\{P_{\neq0}\p^a\Box u^jP_{\neq0}\p^bu^kP_{=0}u^l\\
&\qquad+P_{\neq0}\p^au^jP_{\neq0}\p^b\Box u^kP_{=0}u^l+P_{\neq0}\p^au^jP_{\neq0}\p^bu^kP_{=0}\Box u^l\Big\},
\end{split}
\end{equation}
where the expression of $G_4^i(\p^{\le3}u)$ is similar to that of $\tilde G^i(\p^{\le4}u)$ but
$G_4^i(\p^{\le3}u)$ contains at most third order derivatives of $u$.
Consequently, \eqref{QWE3-1} is is achieved.
\end{proof}

\vskip 0.2 true cm

{\bf \color{blue}{Conflict of Interest Statement:}}

\vskip 0.1 true cm

{\bf The authors declare that there is no conflict of interest in relation to this article.}

\vskip 0.2 true cm
{\bf \color{blue}{Data availability statement:}}

\vskip 0.1 true cm

{\bf  Data sharing is not applicable to this article as no data sets are generated
during the current study.}

\end{document}